\documentclass[11pt,a4paper,final,dvipsnames,reqno]{amsart}

\usepackage{amssymb}
\usepackage{amsmath}
\usepackage[all]{xy}
\usepackage{tikz}
\usepackage{tikz-cd}
\usepackage{caption}
\usepackage[pdftex]{hyperref}
\hypersetup{
    colorlinks=true,
    linkcolor=black,
    urlcolor=black,
    citecolor=black,
    linktocpage=false
}
\usepackage{graphicx}
\usepackage{array}
\usepackage{physics}
\usepackage{bm}
\usepackage{xcolor}
\usepackage{comment}
\usepackage{microtype}
\numberwithin{equation}{section}

\usepackage[left=2.5cm,right=2.5cm,top=3cm,bottom=3cm]{geometry}
\newtheorem{thmIntro}{Theorem}

\newtheorem{theorem}{Theorem}[section]
\newtheorem{lemma}[theorem]{Lemma}
\newtheorem{proposition}[theorem]{Proposition}
\newtheorem{corollary}[theorem]{Corollary}

\theoremstyle{remark}
\newtheorem{remark}[theorem]{Remark}
\theoremstyle{definition}
\newtheorem{definition}[theorem]{Definition}
\newtheorem{example}[theorem]{Example}

\newcommand{\Z}{{\mathbb Z}}
\newcommand{\Q}{{\mathbb Q}}
\newcommand{\R}{{\mathbb R}}
\newcommand{\C}{{\mathbb C}}

\newcommand{\n}{\left\lfloor \frac{n}{2} \right\rfloor}
\newcommand{\nn}{\left\lfloor \frac{n-1}{2} \right\rfloor}

\newcommand{\npeq}{\lfloor \tfrac{n}{2} \rfloor}
\newcommand{\nnpeq}{\lfloor \tfrac{n-1}{2} \rfloor}
\newcommand{\nnnpeq}{\lfloor \tfrac{n^2}{2} \rfloor}
\newcommand{\nnnnpeq}{\lfloor \tfrac{n^2-1}{2} \rfloor}

\DeclareMathOperator{\Id}{Id}

\DeclareMathOperator{\ddet}{\mathbf{det}}
\DeclareMathOperator{\Ker}{Ker}

\DeclareMathOperator{\Stab}{Stab}
\providecommand\sslash{\mathbin{/\mkern-5.5mu/}}

\DeclareMathOperator{\E}{\mathcal{E}}
\DeclareMathOperator{\J}{\mathcal{J}}

\DeclareMathOperator{\GL}{GL}
\DeclareMathOperator{\SL}{SL}
\DeclareMathOperator{\PGL}{PGL}
\DeclareMathOperator{\AGL}{AGL}

\title[Stratification of $\mathrm{AGL}_r(\mathbb{C})$-representation varieties of twisted Hopf links]{Stratification of $\mathrm{AGL}_r(\mathbb{C})$-representation varieties of twisted Hopf links}

\author[\'A. Molina-Navarro]{\'Angel Molina-Navarro}
\address{\'Area de Ciencias de la Computaci\'on y Tecnolog\'ia, Escuela Superior de Ingenier\'ia y Tecnolog\'ia, Universidad Internacional de la Rioja, Avenida de la Paz 137, 26006 La Rioja, Spain.}
\address{Departamento de \'Algebra, Geometr\'ia y Topolog\'ia, Facultad de Ciencias Matem\'aticas, Universidad Complutense de Madrid, Plaza de las Ciencias 3, 28040 Madrid, Spain.}
\email{angmol01\@ucm.es}

\begin{document}

\begin{abstract}
    We provide a stratification of the $\mathrm{AGL}_r(\mathbb{C})$-representation variety of the fundamental group of the complement of a twisted Hopf link in terms of a stratification of the corresponding $\mathrm{GL}_r(\mathbb{C})$-representation variety. For ranks $1$ and $2$, we explicitly describe this stratification and compute the motives of these varieties in terms of the Lefschetz motive $q=[\mathbb{C}]$ in the Grothendieck ring of complex algebraic varieties $K_0(\mathbf{Var}_{\mathbb{C}})$.
\end{abstract}

\maketitle

\begingroup
\addtolength{\parskip}{-0.5em}
\tableofcontents
\endgroup

%%%%%%%%%%%%%%%%%%%%%%
\section{Introduction}
%%%%%%%%%%%%%%%%%%%%%%

The investigation of $3$-manifold topology via algebraic varieties parametrizing group representations was first systematically developed by Culler and Shalen \cite{cullerShalen83}. Given a connected manifold $X$ and a complex linear group $G$, the representation variety of $\pi_1(X)$ into $G$, together with the associated moduli space of such representations, also known as the character variety, encode relevant information about the topology of $X$. In particular, these ideas apply naturally to the study of links, since for any link $L\subseteq\mathbb{S}^3$, one may consider the $G$-representation variety of the $3$-manifold $\mathbb{S}^3-L$.

The representation and character varieties of knots and links have been widely studied in the last fifty years. For instance, geometrical descriptions and motivic invariants of these spaces have been thoroughly investigated for torus knots in \cite{gonzalezLogaresMunoz23, gonzalezMartinezMunoz23, gonzalezMunoz22, gonzalezMunoz22-2, martinezMunoz15, munoz09, munozPorti16}, for the figure-eight knot in \cite{heusenerMunozPorti16}, for twisted Hopf links in \cite{gonzalezLogaresMartinezMunoz24, gonzalezMunoz23}, and, more generally, for torus links in \cite{florentinoLawton25, gonzalezMartinezMunoz24}. Moreover, since the independent work of Burde \cite{burde67} and de Rham \cite{deRham67} relating affine representations of knot and link groups to Alexander invariants, the interplay between the representation theory of such groups and polynomial invariants of knots and links has been an active area of research, cf. \cite{cooperEtAl94, gonzalezMartinezMunoz25, heusenerPortiSuarez01, heusenerPorti05, hironaka97, kitano96, wada94}.

In a different setting, moduli spaces of representations of fundamental groups of complex curves have been thoroughly studied. More precisely, given a smooth complex projective curve $\Sigma$, non-abelian Hodge theory \cite{simpson92} provides a correspondence between the character variety of $\pi_1(\Sigma)$, known as the Betti moduli space, and a moduli space of Higgs bundles. Furthermore, the Betti moduli space is closely related to the existence of geometric structures on the underlying surface of $\Sigma$, see \cite{goldman88}.

Motivated by the computation of motivic invariants, we study the $G$-representation varieties of an infinite family of $3$-manifolds, namely the complements in $\mathbb{S}^3$ of the Hopf link with $n$ twists, for each integer $n\geq 1$, where $G=\AGL_r(\C)$ is the group of affine transformations of the affine space $\C^r$. In particular, for ranks $1$ and $2$, we provide a stratification of these varieties, we express their motives in terms of the Lefschetz motive $q=[\C]$ in the Grothendieck ring of complex algebraic varieties $K_0(\mathbf{Var}_{\C})$, and show that they lie in the subring generated by $q$. We establish the following result by combining Theorem \ref{thm_motiveAGL1} and Theorem \ref{thm_motiveAGL2} in this manuscript.
\begin{thmIntro}\label{intro-thmA}
    Let $n\in\Z$ with $n\geq 1$. The motives of the $\AGL_r(\C)$-representation varieties of the $n$-twisted Hopf link for $r=1,2$ are
    $$ \begin{aligned}
        [\mathfrak{R}_n(\AGL_1(\C))] &= q(q-1)(nq-n+1), \\
        [\mathfrak{R}_n(\AGL_2(\C))] &= q^3(q+1)(q-1)^2\left( \tfrac{(n-1)(n-2)}{2}q^4 + \left( \tfrac{n(n-1)}{2}+\npeq \right)q^3 \right. \\
        & \quad + \left. \left( 4n + \nnpeq - \nnnpeq - 3 \right)q^2 - \left( n^2+\npeq+\nnnnpeq-2 \right)q \right. \\
        & \quad + n^2-3n+\npeq+1 \Big).
    \end{aligned} $$
\end{thmIntro}

Moreover, in light of Section \ref{section_3} and Remark \ref{remark_charVarAGLR}, we obtain the following.
\begin{thmIntro}\label{intro-thmB}
    Let $n\in\Z$ with $n\geq 1$. The $\AGL_r(\C)$-character variety of the $n$-twisted Hopf link is an affine variety which is isomorphic to the associated $\GL_r(\C)$-character variety. For $r=2$, its motive is
    $$ [\mathfrak{M}_n(\AGL_2(\C))] = (q-1)^2\left( \left( \npeq +1 \right)q^2 -\npeq q + n-\npeq \right). $$
\end{thmIntro}
The affineness of the $G$-character variety is not guaranteed in general, since it depends on whether the group $G$ is reductive.

\subsection*{Outline of the paper} Section \ref{section_1} provides a stratification of the $\GL_r(\C)$-representation variety of the twisted Hopf link in terms of Jordan types, which plays a key role throughout the paper. In Section \ref{section_2}, an explicit geometrical description of the aforementioned stratification is given in the case $r=2$. The motive of the representation variety is also computed. Moreover, we make use of this stratification to describe the $\GL_2(\C)$-character variety of the twisted Hopf link in Section \ref{section_3}. In Section \ref{section_4} we present a stratification of the $\AGL_r(\C)$-representation variety of the twisted Hopf link in terms of the previous stratification of the $\GL_r(\C)$-representation variety. Moreover, an explicit description of this stratification and of the motive of the $\AGL_r(\C)$-representation variety is given in the case $r=2$ in Section \ref{section_5}.

\subsection*{Acknowledgments} The author would like to express his sincere gratitude to his Ph.D. advisors, Vicente Mu\~{n}oz and Javier Mart\'inez, for their constant guidance and support throughout the preparation of this work. The author also wishes to thank Enrique Aycart-Maldonado for pointing out an error in the computation of $[\mathfrak{R}_n(\mathrm{AGL}_2(\mathbb{C}))]$ in an earlier version of this manuscript.

%%%%%%%%%%%%%%%%%%%%%%%%%%%%%%%%%%%%
\section{Preliminaries and notation}
%%%%%%%%%%%%%%%%%%%%%%%%%%%%%%%%%%%%

Throughout the paper, all algebraic varieties are considered over the field $\C$ of complex numbers. Accordingly, we omit the base field from the notation for classical algebraic groups (e.g., $\GL_r$, $\SL_r$, $\PGL_r$, $\AGL_r$). Given $n\in\Z$ with $n\geq 1$, denote by $\mu_n$ the set of $n$-th roots of unity $\mu_n=\{\zeta\in\C\:|\:\zeta^n=1\}$, and by $\mu_n^*=\mu_n-\{1\}$.

% % % % % % % % % % % % % % % % % % % % % % % % % %
\subsection{Representation and character varieties}
% % % % % % % % % % % % % % % % % % % % % % % % % %

Let $\Gamma$ be a finitely presented group and let us fix a linear algebraic group $G\subseteq\GL_r$. A \textit{representation} of $\Gamma$ into $G$ is a group homomorphism $\rho:\Gamma\to G$. We denote the set of such representations as $\mathfrak{R}(\Gamma,G)=\mathrm{Hom}_{\mathbf{Grp}}(\Gamma,G)$. Given a presentation $\Gamma=\langle x_1,\dots,x_n \:|\: r_1,\dots, r_s \rangle$, a representation $\rho\in\mathfrak{R}(\Gamma,G)$ is completely determined by the $n$-tuple $(A_1,\dots,A_n)$ satisfying the relations $r_j(A_1,\dots,A_n)=\Id$, where $A_i=\rho(x_i)\in G$. Thus, we can regard the set $\mathfrak{R}(\Gamma,G)$ as a Zariski-closed subset of the affine variety $G^n$. We call this algebraic variety the $G$-\textit{representation variety} of $\Gamma$.

Given a representation $\rho:\Gamma\to G$, a subspace $V\subseteq\C^r$ is called \textit{$\Gamma$-invariant} if $\rho(x)v\in V$ for each $x\in \Gamma$ and $v\in V$. The representation $\rho$ is \textit{irreducible} if the only subspaces of $\C^r$ which are $\Gamma$-invariant are $\{0\}$ and $\C^r$. Otherwise, the representation $\rho$ is called \textit{reducible}. Note that reducibility induces a stratification of the representation variety
\begin{equation}\label{eq_stratificationRedIrr}
    \mathfrak{R}(\Gamma,G)=\mathfrak{R}^{\mathrm{red}}(\Gamma,G) \sqcup \mathfrak{R}^{\mathrm{irr}}(\Gamma,G)
\end{equation}
into locally closed subsets, namely the sets of reducible and irreducible representations.

There is a natural left action of $G$ on the representation variety by conjugation, i.e., $G$ acts on $\mathfrak{R}(\Gamma,G)$ by $(g\cdot\rho)(x)=g^{-1}\rho(x)g$. We say that two representations $\rho,\rho':\Gamma\to G$ are \textit{equivalent} if they lie in the same $G$-orbit. We define the moduli space of representations of $\Gamma$ into $G$, also known as the $G$-\textit{character variety} of $\Gamma$, as the GIT quotient
$$ \mathfrak{M}(\Gamma,G)=\mathfrak{R}(\Gamma,G)\sslash G. $$
Observe that we may quotient by $\mathrm{P}G=G/\mathrm{Z}(G)$ since the center of $G$ acts trivially by conjugation.
\begin{remark}
    If $\rho:\Gamma\to G$ is a reducible representation, there exists a flag of subspaces $\{0\}=V_0\subsetneq V_1\subsetneq \dots \subsetneq V_k=V$ such that, for each $1\leq i\leq n$, the subspace $V_i$ is $\Gamma$-invariant and $\rho$ induces an irreducible representation $\rho_i$ in the quotient $V_i/V_{i-1}$. In this case, the representations $\rho$ and $\widetilde{\rho}=\bigoplus_{i=1}^n\rho_i$ define the same point in the character variety $\mathfrak{M}(\Gamma,G)$, and we say that $\rho$ and $\widetilde{\rho}$ are \textit{$S$-equivalent}.
\end{remark}

If we realize the representation variety as an affine scheme $\mathfrak{R}(\Gamma,G)=\mathrm{Spec}(R)$, where $R$ is a finitely generated and reduced $\C$-algebra, then from the definition of the GIT quotient
$$ \mathfrak{M}(\Gamma,G)=\mathrm{Spec}(R^G), $$
where $R^G$ denotes the subalgebra of $G$-invariants elements of $R$ under the induced action of $G$. If $G$ is reductive, Nagata's theorem \cite[Thm. 3.4]{newstead78} implies that $R^G$ is finitely generated, and hence $\mathfrak{M}(\Gamma,G)$ is an affine variety.

In this manuscript, we focus on the case of the fundamental group of the complement of a link $L\subseteq\mathbb{S}^3$, which is a link invariant. Thus we may define the \textit{$G$-representation variety of the link} $L$
$$ \mathfrak{R}(L,G)=\mathfrak{R}(\pi_1(\mathbb{S}^3-L),G) $$
as the $G$-representation variety of the fundamental group of its complement. In the same vein, we define the \textit{$G$-character variety of the link} $L$ as $\mathfrak{M}(L,G)=\mathfrak{M}(\pi_1(\mathbb{S}^3-L),G)$.

% % % % % % % % % % % % % % % % % %
\subsection{Mixed Hodge structures}
% % % % % % % % % % % % % % % % % %

Given a subring $R$ of the field of real numbers $\R$, let $V$ be a finitely generated $R$-module. Recall that a \textit{pure $R$-Hodge structure} on $V$ is a direct sum decomposition $V_{\C}=V\otimes_{R}\C=\bigoplus_{p,q\in\Z}V^{p,q}$ such that $V^{p,q}=\overline{V^{q,p}}$. Given an integer $k\in\Z$, the \textit{weight $k$ part} $V^k$ is the underlying $R$-module of $\bigoplus_{p+q=k}V^{p,q}$. We say that $V$ carries a \textit{pure $R$-Hodge structure of weight $k$} if $V=V^k$. In this case, the complex vector space $V_{\C}=\bigoplus_{p+q=k}V^{p,q}$ admits a descending filtration $F^{\bullet}=\{F^p\}_{p\in\Z}$, the so-called \textit{Hodge filtration}, defined by $F^p=\bigoplus_{r\geq p} V^{r,k-r}$. Moreover, the associated $p$-th graded piece of the filtration $\mathrm{Gr}^p_{F}=F^p/F^{p+1}$ is given by $\mathrm{Gr}^p_{F}=V^{p,k-p}$. Conversely, any descending filtration $F^{\bullet}$ on $V_{\C}$ such that $F^p \oplus \overline{F^q}=0$ whenever $p+q\neq k$ induces a pure $R$-Hodge structure of weight $k$ on $V$.

A \textit{mixed Hodge structure} on a finitely generated $\Z$-module $V$ consists of two finite filtrations: an ascending filtration $W_{\bullet}$ on $V_{\Q}=V\otimes_{\Z}\Q$, called the \textit{weight filtration}, and a descending filtration $F^{\bullet}$ on $V_{\C}=V\otimes_{\Z}\C$, called the \textit{Hodge filtration}, which induces a pure $\Q$-Hodge structure of weight $k$ on each graded piece $\mathrm{Gr}_k^W \otimes_{\Q}\C$. The complex vector spaces $V^{p,q}=\mathrm{Gr}_F^p\left( \mathrm{Gr}^W_{p+q}(V_{\Q}) \otimes_{\Q} \C \right)$
are the \textit{Hodge pieces} of the mixed Hodge structure and the non-negative integers $ h^{p,q}(V_{\Z})=\mathrm{dim}_{\C}(V^{p,q}) $ are called the \textit{Hodge numbers} of $V$. Notice that any finitely generated $\Z$-module carrying a pure Hodge structure canonically carries a mixed Hodge structure.

Let $X$ be a compact Kähler manifold. Then the cohomology group $H^k(X,\Z)$ carries a pure $\Z$-Hodge structure. In particular, this fact holds for the case of smooth projective varieties. For arbitrary quasi-projective varieties, we have the following fundamental result due to Deligne \cite{deligne71}.
\begin{theorem}[Deligne]
    Let $X$ be a (possibly non-smooth) quasi-projective variety. For each integer $k\in\Z$, the integer cohomology group $H^k(X,\Z)$ and the integer cohomology group with compact support $H^k_c(X,\Z)$ have mixed Hodge structures. Moreover, when $X$ is smooth, the mixed Hodge structure on $H^k(X,\Z)$ is the pure $\Z$-Hodge structure of weight $k$.
\end{theorem}
Given a quasi-projective variety $X$, we define its \textit{Hodge numbers}
$$ h^{k,p,q}_c(X)=h^{p,q}(H^k_c(X,\Z)) $$
as the Hodge numbers of the mixed Hodge structure on the integer cohomology with compact support. Furthermore, we call
$$ e(X)(u,v)=\displaystyle\sum_{k,p,q\in\Z}(-1)^kh^{k,p,q}_c(X)u^pv^q \in \Z[u^{\pm 1},v^{\pm 1}] $$
the \textit{Hodge-Deligne polynomial} or \textit{$E$-polynomial} of $X$. If $h^{k,p,q}_c(X)=0$ unless $p=q$, then the $E$-polynomial $e(X)$ depends only on the single variable $q=uv$. In this case, we say that the variety $X$ is \textit{of balanced type}. All the varieties considered in this manuscript satisfy this condition. We record some examples together with their $E$-polynomials:
\begin{itemize}
    \item $e(\C^n)=q^n$,
    \item $e(\C^*)=q-1$,
    \item $e(\GL_n)=(q^n-1)(q^n-q)\dots(q^n-q^{n-1})$,
    \item $e(\SL_n)=e(\PGL_n)=q^{n-1}(q^n-1)(q^n-q)\dots(q^n-q^{n-2})$.
\end{itemize}
In particular, $e(\GL_2)=q(q+1)(q-1)^2$ and $e(\SL_2)=e(\PGL_2)=q(q+1)(q-1)$.

One of the most remarkable properties of the $E$-polynomial is its additivity with respect to stratifications, i.e., if $X$ is a quasi-projective variety which can be decomposed into locally closed subsets $X=\bigsqcup_{i=1}^nX_i$, then $e(X)=\sum_{i=1}^ne(X_i)$. Moreover, if $Y$ is another quasi-projective variety, we have $e(X\times Y)=e(X)e(Y)$. This multiplicativity property can be generalized for certain fibre bundles, as shown in \cite[Prop. 2.4]{logaresMunozNewstead13}.
\begin{theorem}[Logares--Mu\~{n}oz--Newstead]
    Let $F$, $X$ and $B$ be smooth quasi-projective varieties, with $B$ connected. Suppose that $p:X\to B$ is a morphism of algebraic varieties which is also an analytic bundle with fibre $F$ (not necessarily locally trivial in the Zariski topology), and assume that the associated monodromy action is trivial. Then $e(X)=e(F)e(B)$.\label{thm_LogaresMunozNewstead}
\end{theorem}

% % % % % % % % % % % % % % % % % % % % % %
\subsection{The equivariant $E$-polynomial}
% % % % % % % % % % % % % % % % % % % % % %

Throughout the paper, we will need to compute the $E$-polynomial of certain algebraic varieties that are equipped with an action of a finite group. To this end, we recall the equivariant $E$-polynomial.

Consider a quasi-projective variety $X$ endowed with an action of a finite group $\mathsf{F}$. The induced $\mathsf{F}$-action on the $k$-th cohomology group with compact support $H^k_c(X,\Z)$ is compatible with the mixed Hodge structure, so if we denote by $H_c^{k,p,q}(X)$ the associated $(p,q)$-Hodge piece, it defines a class in the representation ring $R(\mathsf{F})$ of $\mathsf{F}$. We define the \textit{equivariant $E$-polynomial} of $X$ as
$$ e_{\mathsf{F}}(X)(u,v)=\sum_{k,p,q\in\Z}(-1)^k[H_c^{k,p,q}(X)]u^pv^q\in R(\mathsf{F})[u^{\pm1},v^{\pm1}]. $$
The classical $E$-polynomial may be recovered from the equivariant version; see, e.g., \cite[Prop. 2.2]{calleja24}. Furthermore, a result analogous to Theorem \ref{thm_LogaresMunozNewstead} holds in the equivariant case, cf. \cite{calleja24,florentinoNozadZamora21}.

We shall focus on the case $\mathsf{F}=S_2$. Here we have two different irreducible representations, namely the trivial representation $T$ and the sign representation $N$. If we write $e_{S_2}(X)=aT+bN$, then $a+b=e(X)$ and $a=e(X/S_2)$. Moreover, if $Y$ is another quasi-projective variety equipped with an $S_2$-action and $e_{S_2}(Y)=cT+dN$, we have
$$ e_{S_2}(X\times Y)=(ac+bd)T+(ad+cb)N. $$

\begin{example}\label{lemma_actionOnPGL2}
    Consider the action of the group $S_2$ on the algebraic group $\PGL_2$ given by switching rows, i.e., $S_2$ acts by left multiplication by the matrix $P_0=\left( \begin{smallmatrix}
        0 & 1 \\ 1 & 0
    \end{smallmatrix} \right)$. If we denote by $\mathcal{D}\subseteq\PGL_2$ the space of diagonal matrices, we may consider the induced $S_2$-action on the quotient $\PGL_2/\mathcal{D}$. In this case, we have that $e(\PGL_2/\mathcal{D})=q^2+q$ and $e((\PGL_2/\mathcal{D})/S_2)=q^2$, so $e_{S_2}(\PGL_2/\mathcal{D})=q^2T+qN$. For further details, see \cite[Section 5]{gonzalezMunoz23}.
\end{example}

The following result will be needed later, and its proof can be found in \cite[Prop. 2.2]{gonzalezMunoz23}.
\begin{lemma}\label{lemma_finiteGroupAction}
    Let $G$ be a connected algebraic group equipped with an action of a finite group $\mathsf{F}$ acting by inner automorphisms. Then
    $$ e_{\mathsf{F}}(G)=e(G)T, $$
    where $T$ denotes the trivial representation of the representation ring $R(\mathsf{F})$.
\end{lemma}

% % % % % % % % % % % % % % % % % % % % % % % % % % % % %
\subsection{The Grothendieck ring of algebraic varieties}
% % % % % % % % % % % % % % % % % % % % % % % % % % % % %

Consider the category $\mathbf{Var}_{\C}$ of complex algebraic varieties together with regular maps. We may define its \textit{Grothendieck ring}, denoted $K_0(\mathbf{Var}_{\C})$, to be the commutative ring whose underlying abelian group is the free abelian group generated by isomorphism classes of algebraic varieties, subject to the relation $[X]=[Y]+[X-Y]$ for any closed subvariety $Y\subseteq X$, and whose multiplication is given by the cartesian product of varieties $[X]\cdot[Y]=[X\times Y]$. More generally, if $p:X\to B$ is a locally trivial fibre bundle in the Zariski topology with fibre $F$, then $[X]=[F]\cdot[B]$. For this reason, the $E$-polynomial induces a ring homomorphism
$$ e:K_0(\mathbf{Var}_{\C}) \longrightarrow \Z[u^{\pm 1},v^{\pm 1}] $$
which factors through the $K_0$ group of the abelian category of complex mixed Hodge structures $\mathbf{MHS}_{\C}$, cf. \cite{deligne71,gonzalezLogaresMunoz21}.

Given a variety $X\in\mathbf{Var}_{\C}$, its isomorphism class $[X]\in K_0(\mathbf{Var}_{\C})$ is called the \textit{virtual class} or the $\textit{motive}$ of $X$. The class of the affine line, usually denoted by $q=[\C]$ or $\mathbb{L}=[\C]$, is called the \textit{Lefschetz motive}.

\begin{remark}
    The notation $q=[\C]$ for the Lefschetz motive is justified by the fact that the $E$-polynomial of $\C$ is precisely $e(\C)=q$, where $q=uv$. Moreover, if $X\in\mathbf{Var}_{\C}$ is a variety whose motive lies in the subring of $K_0(\mathbf{Var}_{\C})$ generated by the Lefschetz motive, then the $E$-polynomial of $X$ coincides with its motive, regarded as a polynomial in $q$.
\end{remark}

\begin{remark}\label{remark_equivariantMotives}
    There is also an equivariant version of motives, cf. \cite{vogel24}. Moreover, most computations for equivariant $E$-polynomials parallel those for the corresponding equivariant motives. For this reason, we restrict ourselves to the case of equivariant $E$-polynomials when our variety is equipped with a group action.
\end{remark}

% % % % % % % % % % % % % % % %
\subsection{Twisted Hopf links}
% % % % % % % % % % % % % % % %

Let $H\subseteq \mathbb{S}^3$ be the \textit{Hopf link} (see \cite[Chapter 1]{cromwell04}), which is the simplest non-trivial link with $2$ components. Twisting its components $n$ times yields a $2$-component link with $2n$ crossings, as depicted in Figure \ref{figure_twistedHopfLink}, called the \textit{Hopf link with $n$ twists}, or simply the \textit{$n$-twisted Hopf link}, usually denoted by $H_n\subseteq\mathbb{S}^3$. 

\begin{figure}[ht]
    \begin{center}
    \includegraphics[width=4.5cm]{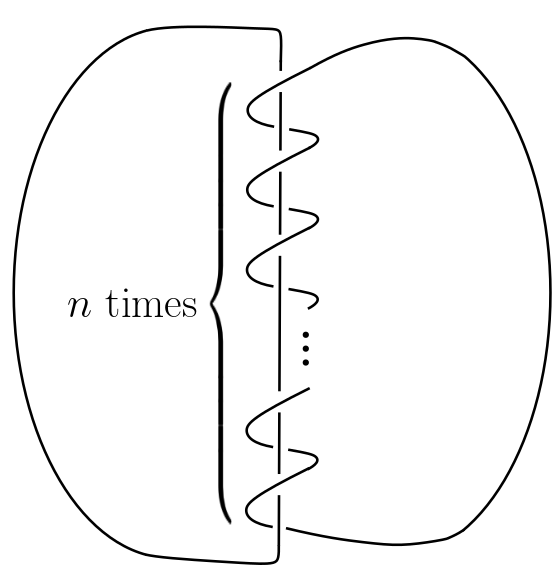}
    \caption{Hopf link with $n$ twists, reproduced from \cite{gonzalezMunoz23}.}\label{figure_twistedHopfLink}
    \end{center}
\end{figure}

For the link group of $H_n$, we have the following result \cite[Prop. 3.1]{gonzalezMunoz23}.

\begin{proposition}\label{propo_presentationHopfGroup}
    Let $n\in\Z$ with $n\geq 1$ and let $\Gamma_n=\pi_1(\mathbb{S}^3-H_n)$ denote the fundamental group of the complement of the $n$-twisted Hopf link. Then $\Gamma_n$ has the presentation
    $$ \Gamma_n=\langle a,b \:|\: [a^n,b]=1 \rangle, $$
    where $[\cdot,\cdot]$ denotes the group commutator.
\end{proposition}

\begin{remark}
    For $n=1$, the fundamental group of the complement of the Hopf link is isomorphic to the fundamental group of the torus $T^2=\mathbb{S}^1\times\mathbb{S}^1$.
\end{remark}

Given a linear algebraic group $G\subseteq\GL_r$, we may consider the $G$-representation variety of the $n$-twisted Hopf link as the algebraic variety
\begin{equation}\label{eq_repVarietyHopfLink}
    \mathfrak{R}(H_n,G)=\{ (A,B)\in G^2\:|\:[A^n,B]=\Id_r \}.
\end{equation}
In what follows, we adopt the simplified notation $\mathfrak{R}_n(G)=\mathfrak{R}(H_n,G)$ and $\mathfrak{M}_n(G)=\mathfrak{M}(H_n,G)$ for the $G$-representation variety and the $G$-character variety of the $n$-twisted Hopf link, respectively.

%%%%%%%%%%%%%%%%%%%%%%%%%%%%%%%%%%%%%%%%%%%%%%%%%%%%%%%%%%%%%%%%%%%%%%%%%%%%%%%%%%%%%%%%%
\section{The $\mathrm{GL}_r(\mathbb{C})$-representation variety of the twisted Hopf link}\label{section_1}
%%%%%%%%%%%%%%%%%%%%%%%%%%%%%%%%%%%%%%%%%%%%%%%%%%%%%%%%%%%%%%%%%%%%%%%%%%%%%%%%%%%%%%%%%

Let us describe the representation variety of the $n$-twisted Hopf link (\ref{eq_repVarietyHopfLink}) into the general linear group $G=\GL_r$. To this end, we will mimic the combinatorial and geometric framework developed in \cite[Section 4]{gonzalezMunoz23}, \textit{mutatis mutandis}.

% % % % % % % % % % % % % % % % %
\subsection{Combinatorial setting}
% % % % % % % % % % % % % % % % %

Fix $r \geq 1$. We denote by $\E_r$ the space of possible eigenvalues of a matrix of $\GL_r$, namely $(\C^*)^r$. Given an equivalence relation (or, equivalently, a partition) $\sigma$ on the set $\{1,\dots,r\}$, denote by $\E_r^\sigma \subseteq \E_r$ the family of $(\lambda_1, \ldots, \lambda_r)\in\E_r$ such that $\lambda_i = \lambda_j$ if and only if $i \sim_\sigma j$. If $\sigma = \{v_1, \ldots, v_s\}$, the group $S_\sigma := S_{t_1} \times \ldots \times S_{t_r}$ acts on $\E^\sigma_r$ by permutation of blocks, where $t_i=|\{v_k\in\sigma\::\:|v_k|=i\}|$.

\begin{example}
    Let $r=7$ and consider the partition $\sigma = \{\{1,2,3\}, \{4\}, \{5,6,7\}\}$. In this case, $\E_7^{\sigma}\cong\E_3$. Moreover, there is a natural action of $S_{\sigma}=S_2$ on $\E_3$ given by $(\lambda_1,\lambda_2,\lambda_3) \mapsto (\lambda_3,\lambda_2,\lambda_1)$. In general, if $\sigma$ is an arbitrary partition of $\{1,\dots,r\}$ with $|\sigma|=s\leq r$, then $\E_r^{\sigma}\cong\E_s$.
\end{example}

\begin{definition}
    Given two partitions $\sigma'$ and $\sigma$ of the set $\{1,\dots,r\}$, we say that $\sigma'$ is a \textit{refinement} of $\sigma$, denoted $\sigma'\preceq\sigma$, if the partition $\sigma'$ is obtained from further subdividing $\sigma$.
\end{definition}

Note that, given a refinement $\sigma'\preceq \sigma$, there exists a maximal subgroup $S_{\sigma',\sigma}$ of $S_{\sigma'}$ consisting of the permutations that respect the refinement.

For $n\geq 1$, we consider the map $p_n: \E_r \to \E_r$ defined by $p_n(\lambda_1, \ldots, \lambda_r) = (\lambda_1^n, \ldots, \lambda_r^n)$. Given two partitions $\sigma$ and $\sigma'$, we define the subset
\begin{equation}
    \E_r^{\sigma',\:\sigma} = \E_r^{\sigma'} \cap \: p_n^{-1}(\E_r^{\sigma}).
\end{equation}
whenever $\sigma'\preceq \sigma$. Notice that there is a natural action of $S_{\sigma',\:\sigma}$ on $\mathcal{E}_r^{\sigma',\:\sigma}$.

\begin{example}
    Let $r=2$ and consider the two partitions $\sigma = \{\{1,2\}\}$ and $\sigma'=\{\{1\},\{2\}\}$, where $\sigma'\preceq\sigma$. We have that $\E_2^{\sigma',\:\sigma} = \{ (\lambda_1,\lambda_2)\in(\C^*)^2 : \lambda_1\neq \lambda_2, \lambda_1^n=\lambda_2^n \}$,
    so $\lambda_2=\epsilon\lambda_1$, where $\epsilon\in\mu_n^*$. Therefore, $\E_2^{\sigma',\:\sigma}\cong\C^*\times\mu_n^*$. In this case, $S_{\sigma',\:\sigma}= S_2$ acts by swapping coordinates.
\end{example}

% % % % % % % % % % % % % % %
\subsection{Geometric setting}
% % % % % % % % % % % % % % %

The preceding combinatorial framework may be used to stratify the $\GL_r$-representation variety of the $n$-twisted Hopf link. In order to analyze the possible Jordan forms of the matrices involved, we recall the following definition \cite[Def. 4.2]{gonzalezMunoz23}.

\begin{definition}
    A \textit{Jordan type} of rank $r \geq 1$ is a pair $\xi = (\sigma, \kappa)$, where $\sigma = \{v_1, \ldots, v_s\}$ is a partition of $\{1, \ldots, r\}$ and $\kappa = \{\tau_1, \ldots, \tau_s\}$ is a collection where each 
    $\tau_i$ is a partition of $v_i$ into totally ordered sets. A Jordan type $\xi' = (\sigma', \kappa')$ \textit{refines} $\xi = (\sigma, \kappa)$, denoted by $\xi' \preceq \xi$, if $\sigma'$ refines $\sigma$ and for any $v_i \in \sigma$ that decomposes as $v_i = v_{i_1}' \cup \ldots \cup v_{i_t}'$ in $\sigma'$, we have that $\tau_i = \tau'_{i_1} \cup \ldots \cup \tau'_{i_t}$.
\end{definition}

Roughly speaking, the notion of Jordan type encodes the block structure of the Jordan form of a certain matrix. For further details, we refer the reader to \cite[Section 4]{gonzalezMunoz23}.

\begin{example}\label{example_JordanType}
    Let $r=5$. Consider the Jordan type $\xi=(\sigma,\kappa)$, where $\sigma=\{ \{1,2\}, \{3,4\}, \{5\} \}$ and $\kappa=\{ \{(1,2)\}$, $\{(3,4) \}, \{(5)\} \}$. Hence, the type $\xi$ corresponds to a Jordan matrix of the form
    $$ \begin{pmatrix} 
    \lambda_1 & 0 & 0 & 0 & 0 \\
    1 & \lambda_1 & 0 & 0 & 0 \\
    0 & 0 & \lambda_2 & 0 & 0 \\
    0 & 0 & 1 & \lambda_2 & 0 \\
    0 & 0 & 0 & 0 & \lambda_3 \end{pmatrix}, $$
    where $\lambda_i\neq\lambda_j$ for $i\neq j$.
\end{example}

The symmetric group $S_r$ acts on the set of Jordan types of rank $r$ as follows. Given a Jordan type $\xi=(\sigma,\kappa)$ and a permutation $f\in S_r$, the Jordan type $f\cdot\xi$ is obtained by applying the permutation $f$ to both $\sigma$ and $\kappa$. In this situation, we will say that two Jordan types are \textit{equivalent} if they belong to the same $S_r$-orbit.

Given a Jordan type $\xi=(\sigma,\kappa)$ of rank $r$, denote by $\J_\xi$ the collection of Jordan matrices of $\GL_r$ with ones as subdiagonal elements according to $\kappa$ whose block structure is determined by $\xi$. Allowing arbitrary non-zero entries in the positions prescribed by $\kappa$ yields the space of \textit{generalized} Jordan matrices, denoted ${\J}^g_\xi$. Set $\widetilde{\J}_{\xi} = \GL_r \cdot \J_\xi$, where the $\GL_r$-action is given by conjugation.

\begin{remark}\label{remark_stratificationGL}
    Choosing representatives $\xi_0, \ldots, \xi_{N}$ of the $S_r$-orbits of Jordan types of rank $r$, we obtain a stratification $\GL_r =  \widetilde{\J}_{\xi_0}\sqcup\cdots\sqcup\widetilde{\J}_{\xi_N}$.
\end{remark}

As in the combinatorial setting, we consider the map $p_n:\GL_r\to\GL_r$ given by $p_n(A) = A^n$
and we fix the notation 
$$ \begin{aligned}
    & {\J}_{\xi',\:\xi} = \J_{\xi'} \cap\: p_n^{-1}(\J^g_{\xi}), \\ 
    & \widetilde{\J}_{\xi',\:\xi} = \GL_r \cdot {\J}_{\xi',\:\xi}. 
\end{aligned} $$

Given a Jordan type $\xi=(\sigma,\kappa)$, let $S_\xi$ be the subgroup of $S_\sigma$ consisting of permutations of the blocks of $\sigma$ whose decomposition with respect to $\kappa$ are equivalent. If $\xi' \preceq \xi$ are two Jordan types, we define $S_{\xi',\,\xi}$ to be the subgroup of $S_{\xi'}$ of permutations of $\xi'$ that induce permutations of $\xi$.

\begin{example}
    Let $r=5$. Suppose we have the Jordan types $\xi'=(\sigma',\kappa')$, where $\sigma'=\{  \{1,2\}, \{3,4\}, \{5\} \}$ and $\kappa'=\{ \{(1,2)\},\{(3,4)\}, \{(5)\} \}$; and $\xi=(\sigma,\kappa)$, where $\sigma=\{ \{1,2,3,4\}, \{5\} \}$ and $\kappa=\{ \{(1,2), (3,4) \}, \{(5)\}\}$. Since $\xi'\preceq\xi$, the set $\mathcal{J}_{\xi',\:\xi}$ consists of matrices of the form
    $$ \begin{pmatrix} 
    \lambda_1 & 0 & 0 & 0 & 0 \\
    1 & \lambda_1 & 0 & 0 & 0 \\ 
    0 & 0 & \epsilon\lambda_1 & 0 & 0 \\
    0 & 0 & 1 & \epsilon\lambda_1 & 0 \\ 
    0 & 0 & 0 & 0 & \lambda_2 \end{pmatrix}, $$
    where $\epsilon\in\mu_n^*$ and $\lambda_2\neq\varepsilon\lambda_1$ for each $\varepsilon\in\mu_n$. The group $S_{\xi'}=S_2$ acts by permuting $\{1,2\}$ and $\{3,4\}$, whereas $S_{\xi',\:\xi }=\{1\}$.
\end{example}

The following lemma relates the combinatorial and geometric settings. A detailed proof can be found in \cite[Lemma 4.6]{gonzalezMunoz23}.

\begin{lemma} \label{lem:A}
For any Jordan types $\xi = (\sigma, \kappa)$ and $\xi' = (\sigma', \kappa')$, we have that $\J_{\xi',\:\xi} = \emptyset$ if $\xi'$ does not refine $\xi$, and 
$\J_{\xi',\:\xi} \cong \E_r^{\sigma',\:\sigma}$ if it does. 
\end{lemma}

% % % % % % % % % % % % % % % % % % % % % % % % % % % % %
\subsection{Stratification of the representation variety} 
% % % % % % % % % % % % % % % % % % % % % % % % % % % % %

Given two Jordan types $\xi'$ and $\xi$ such that $\xi' \preceq \xi$,  consider the set
\begin{equation}\label{eq_stratumJordanTypeRepVarGL}
    \mathfrak{R}_n(\GL_r)_{\xi',\:\xi} = \left\{(A, B) \in \widetilde{\J}_{\xi',\:\xi} \times \GL_r \:|\: [A^n,B] = \Id_r \right\}.
\end{equation}
Due to Remark \ref{remark_stratificationGL}, there is a stratification of the $\GL_r$-representation variety of the $n$-twisted Hopf link given by
\begin{equation}\label{eq_stratificationRepVarGLr}
    \mathfrak{R}_n(\GL_r) = \bigsqcup_{\xi',\:\xi} \mathfrak{R}_n(\GL_r)_{\xi',\:\xi},
\end{equation}
where $\xi'$ and $\xi$ range over a set of non-equivalent representatives of all Jordan types of rank $r$. The following proposition describes explicitly the sets (\ref{eq_stratumJordanTypeRepVarGL}). The proof is formally identical to that of \cite[Prop. 4.7]{gonzalezMunoz23}, except that the algebraic group $\SL_r$ (and consequently its conjugation action) is replaced by $\GL_r$.

\begin{proposition}\label{propo_strataRepVar}
    For any Jordan types $\xi'$ and $\xi$ such that $\xi'\preceq\xi$, we have that
    $$ \mathfrak{R}_n(\GL_r)_{\xi',\:\xi} \cong  
	\left.\left(\J_{\xi',\:\xi} \times \big(\PGL_{r}/ \Stab(\xi')\big) \times \widetilde{\Stab}(\xi) \right) \right/ S_{\xi',\:\xi}, $$
    where $\Stab(\xi')$ and $\widetilde{\Stab}(\xi)$ denote the stabilizers in $\PGL_r$ and $\GL_r$, respectively, of a matrix of $\J_{\xi'}$ and $\J_{\xi}$.
\end{proposition}

\begin{remark}
    The $S_{\xi',\:\xi}$-action on $\mathfrak{R}_n(\GL_r)_{\xi',\:\xi}$ that appears in Proposition \ref{propo_strataRepVar} is defined as follows. Since the action of an element $f\in S_{\xi',\:\xi}$ on a matrix $A_0\in \J_{\xi',\:\xi}$ produces a matrix equivalent to $A_0$, there exists $P_f\in\GL_r$ such that $f\cdot A_0=P_f^{-1}A_0 P_f$. Therefore, the $S_{\xi',\:\xi}$-action is given by $f\cdot (A_0,P,B_0)=(P_f^{-1}A_0 P_f,P_f^{-1}P,P_f^{-1}B_0 P_f)$.
\end{remark}

% % % % % % % % % % % % % % % % % % % %
\subsection{Irreducible representations}
% % % % % % % % % % % % % % % % % % % %

Since we already have a stratification of the representation variety $\mathfrak{R}_n(\GL_r)$ in terms of the different Jordan types and their degenerations, we now aim to describe the set of irreducible representations of the $n$-twisted Hopf link into $\GL_r$. We begin with a simple but useful result.

\begin{lemma}\label{lemma_irrRep}
    Let $\rho=(A,B)\in\mathfrak{R}_n(\GL_r)$ be an irreducible representation. Then, $A^n$ is a non-zero multiple of the identity.
\end{lemma}
\begin{proof}
    Since $\rho$ is irreducible and $A^n$ commutes with $\rho$, Schur's lemma yields the result.
\end{proof}

The next result can now be established.

\begin{proposition}
    The only strata $\mathfrak{R}_n(\GL_r)_{\xi_j\:,\:\xi_i}$ that may contain irreducible representations are those of the form $\mathfrak{R}_n(\GL_r)_{\xi\:,\:\xi_0}$, where $\xi_0$ corresponds to multiples of the identity and $\xi\neq\xi_0$ is a diagonalizable Jordan type.
\end{proposition}
\begin{proof}
    Lemma \ref{lemma_irrRep} ensures that if $\xi_i\neq\xi_0$, then $\mathfrak{R}_n(\GL_r)_{\xi_j\:,\:\xi_i}$ contains only reducible representations. If $\rho=(A,B)\in\mathfrak{R}_n(\GL_r)_{\xi_0\:,\:\xi_0}$, then any subspace $V\subsetneq\C^r$ such that $B(V)=V$ also satisfies $A(V)=V$, since $A$ is a scalar multiple of the identity. Therefore, $\rho$ is reducible.
\end{proof}

\begin{remark}\label{remark_irreducibleLocus}
    Note that the above result ensures that the irreducible locus of the $\GL_r$-representation variety decomposed into different strata 
    $$ \mathfrak{R}_n^{\mathrm{irr}}(\GL_r)=\bigsqcup_{\xi\:\in\:\Xi-\{\xi_0\}}\mathfrak{R}_n^{\mathrm{irr}}(\GL_r)_{\xi\:,\:\xi_0}, $$
    where $\Xi$ is a set of non-equivalent representatives of the diagonalizable Jordan types of rank $r$.
\end{remark}

% % % % % % % % % % % % % % % % % % % % % % % % % % % % % % % % % % % % % % % % % % % % % % % % % % % % % % % % % % % % %
\subsection{Representations of the twisted Hopf link group in $\mathrm{SL}_r(\mathbb{C})$ and $\mathrm{GL}_r(\mathbb{C})$}
% % % % % % % % % % % % % % % % % % % % % % % % % % % % % % % % % % % % % % % % % % % % % % % % % % % % % % % % % % % % %

Let us study the relationship between the $G$-representation varieties of the $n$-twisted Hopf link for $G=\SL_r,\GL_r$. For this purpose, we generalize the known relation between the representation varieties of knot groups into these linear groups \cite[Lemma 2.2]{heusenerMunozPorti16} to the case of the fundamental group of the complement of the $n$-twisted Hopf link.

Consider the determinantal map
$$ \ddet : \mathfrak{R}_n(\GL_r) \longrightarrow \mathfrak{R}_n(\GL_1)=(\C^*)^2 $$
given by $\ddet(A,B)=(\det(A),\det(B))$, which is a morphism of complex affine varieties. Note that $\mathfrak{R}_n(\SL_r)=\ddet^{-1}(1,1)$.

There is a natural $(\C^*)^2$-action on $\mathfrak{R}_n(\GL_r)$ given by componentwise multiplication, that is, given by $(\lambda_1,\lambda_2)\cdot (A,B)=(\lambda_1 A,\lambda_2 B)$. Indeed, this action is well-defined since
$$ \begin{aligned}
    [(\lambda_1A)^n,\lambda_2B] & = (\lambda_1A)^n(\lambda_2B)(\lambda_1A)^{-n}(\lambda_2B)^{-1} = A^nBA^{-n}B^{-1} = [A^n,B] = \Id_r. 
\end{aligned} $$

Regarding the determinantal map, we have that
$$ \ddet((\lambda_1,\lambda_2)\cdot(A,B))=\ddet(\lambda_1 A,\lambda_2 B)=(\lambda_1^r \det(A),\lambda_2^r \det(B)), $$
so the equality $\ddet((\lambda_1,\lambda_2)\cdot(A,B))=\ddet(A,B)$ holds if and only if $\lambda_1,\lambda_2\in\mu_r$. Therefore, if we consider the surjective map
$$ (\C^*)^2 \times \mathfrak{R}_n(\SL_r) \longrightarrow \mathfrak{R}_n(\GL_r) $$
given by $((\lambda_1,\lambda_2),(A,B))\mapsto (\lambda_1 A,\lambda_2 B)$, note that two points are mapped to the same representation if and only if $\lambda_1,\lambda_2\in\mu_r$, so there is an isomorphism
\begin{equation*}
    \mathfrak{R}_n(\GL_r) \cong \left.\left( \mathfrak{R}_n(\SL_r) \times (\C^*)^2 \right) \right/ \mu_r^2.
\end{equation*}
Since the $\PGL_r$-action by conjugation on $\mathfrak{R}_n(\SL_r)\times(\C^*)^2$ commutes with the $\mu_r^2$-action, 
\begin{equation*}
    \begin{aligned}
    \mathfrak{M}_n(\GL_r)=\mathfrak{R}_n(\GL_r)\sslash \PGL_r & \cong \left.\left( \left( \mathfrak{R}_n(\SL_r)\times(\C^*)^2 \right) \sslash \PGL_r \right) \right/  \mu_r^2 \\
    & \cong \left.\left( \mathfrak{M}_n(\SL_r)\times(\C^*)^2 \right) \right/ \mu_r^2.
\end{aligned}
\end{equation*}

We summarize the above discussion in the next result.

\begin{proposition}\label{propo_relationRepCharSLandGL}
    There are isomorphisms of affine varieties
    $$ \begin{aligned}
        \mathfrak{R}_n(\GL_r) & \cong \left( \mathfrak{R}_n(\SL_r) \times (\C^*)^2 \right)/\mu_r^2, \quad
        \mathfrak{M}_n(\GL_r) \cong \left( \mathfrak{M}_n(\SL_r) \times (\C^*)^2 \right)/\mu_r^2;
    \end{aligned} $$
    where $\mu_r^2=\mu_r\times\mu_r$ acts by componentwise multiplication.
\end{proposition}

%%%%%%%%%%%%%%%%%%%%%%%%%%%%%%%%%%%%%%%%%%%%%%%%%%%%%%%%%%%%%%%%%%%%%%%%%%%%%%%%%%%%%%%%%%%%%%%%%%%%%%%%%%%
\section{Stratification of the $\mathrm{GL}_2(\mathbb{C})$-representation variety of the twisted Hopf link}\label{section_2}
%%%%%%%%%%%%%%%%%%%%%%%%%%%%%%%%%%%%%%%%%%%%%%%%%%%%%%%%%%%%%%%%%%%%%%%%%%%%%%%%%%%%%%%%%%%%%%%%%%%%%%%%%%%

We provide a precise geometrical description of the stratification of the $\GL_2$-representation variety of the $n$-twisted Hopf link introduced in the previous section.

% % % % % % % % % % % % % % % % % % % % % % % % % % % % %
\subsection{Combinatorial and geometric analysis}
% % % % % % % % % % % % % % % % % % % % % % % % % % % % %

For the case $r=2$, we have two distinct partitions, namely, $\sigma_0=\{ \{1,2\} \}$ and $\sigma_1=\{ \{1\} , \{2\} \}$. The former corresponds to matrices of $\GL_2$ with equal eigenvalues, while the latter corresponds to matrices with different eigenvalues. Therefore, we have the following spaces of eigenvalues:
\begin{equation}\label{eq_eigenvalueSpaces}
    \begin{aligned}
    \E_2^{\sigma_0} & = \{ (\lambda_1,\lambda_2)\in(\C^*)^2 \:|\: \lambda_1=\lambda_2 \} \cong \C^*, \\
    \E_2^{\sigma_1} & = \{ (\lambda_1,\lambda_2)\in(\C^*)^2 \:|\: \lambda_1\neq\lambda_2 \} = \E_2-\E_2^{\sigma_0} \cong (\C^*)^2-\C^*, \\
    \E_2^{\sigma_0,\:\sigma_0} & = \E_2^{\sigma_0} \cong \C^*, \\
    \E_2^{\sigma_1,\:\sigma_0} & = \{ (\lambda_1,\lambda_2)\in(\C^*)^2 \:|\: \lambda_1\neq\lambda_2, \lambda_1^n=\lambda_2^n \} \\
    & = \{ (\lambda_1,\lambda_2)\in(\C^*)^2 \:|\: \lambda_2=\epsilon\lambda_1, \text{ where } \epsilon\in\mu_n^* \} \cong \C^*\times\mu_n^*, \\
    \E_2^{\sigma_1,\:\sigma_1} & = \{ (\lambda_1,\lambda_2)\in(\C^*)^2 \:|\: \lambda_1\neq\lambda_2, \lambda_1^n\neq \lambda_2^n \} \cong \E_2^{\sigma_1}-\E_2^{\sigma_1,\:\sigma_0}.
    \end{aligned}
\end{equation}

We have to analyze the $S_2$-action on the different strata. Recall that, for the action of $S_2$ on $(\C^*)^2$ given by swapping coordinates, we have that $(\C^*)^2/S_2\cong\C\times\C^*$. The isomorphism is given by $(\lambda_1,\lambda_2)\mapsto (s_1,s_2)$, where $s_1=\lambda_1+\lambda_2\in\C$ and $s_2=\lambda_1\lambda_2\in\C^*$, which are the coefficients of the characteristic polynomial of a matrix with eigenvalues $\lambda_1$ and $\lambda_2$.

\begin{lemma}\label{lemma_actionS2}
    Consider the $S_2$-action on $\E_2$ given by swapping coordinates.
    \begin{itemize}
        \item For the induced action on $\E_2^{\sigma_1}$, we have that $\E_2^{\sigma_1}/S_2\cong (\C\times\C^*)-\C^*$.
        \item For the induced action on $\E_2^{\sigma_1,\:\sigma_0}$, we have that $\E_2^{\sigma_1,\:\sigma_0}/S_2\cong \C^*\times F_n$, where $F_n$ is a finite set of $\lfloor\frac{n}{2}\rfloor$ points. 
    \end{itemize}
\end{lemma}
\begin{proof}
    If we consider the induced $S_2$-action on the diagonal $\E_2^{\sigma_1}$, a pair $(\lambda,\lambda)$ is mapped to $(2\lambda,\lambda^2)$, i.e., the set $\E_2^{\sigma_0}$ is mapped to the set $\{(s_1,s_2)\in\C\times\C^* \:|\: s_1^2-4s_2=0 \}$. Thus, we have the first isomorphism
    $$ \E_2^{\sigma_1}/S_2\cong \{(s_1,s_2)\in\C^*\times\C \:|\: s_1^2-4s_2\neq 0\}\cong(\C\times\C^*)-\C^*. $$

    For the second isomorphism, the $S_2$-action on $\E_2^{\sigma_1,\:\sigma_0}$ given by swapping coordinates can be thought as the $S_2$-action on $\C^*\times\mu_n^*$ given by $(\lambda,\epsilon)\mapsto (\epsilon\lambda,\epsilon^{-1})$. Thus, the orbit of a point $(\lambda,\epsilon)\in\E_2^{\sigma_1,\:\sigma_0}$ is given by $S_2\cdot (\lambda,\epsilon)=\{(\lambda,\epsilon),(\epsilon\lambda,\epsilon^{-1})\}$. If $\epsilon\in\mu_n^*$ satisfies $\epsilon\neq -1$, which always holds when $n$ is odd, then the $S_2$-action swaps the punctured lines $\C^*\times\{\epsilon\}$ and $\C^*\times\{\epsilon^{-1}\}$, so they correspond to a single punctured line in the quotient. When $n$ is even, in addition to the previous contributions, the punctured line $\C^*\times\{-1\}$ is $S_2$-invariant and the quotient $(\C^*\times\{-1\})/S_2$ is isomorphic to the punctured line $\C^*\times\{-1\}$ via the map $S_2\cdot(\lambda,-1)\mapsto (\lambda^2,-1)$. Thus, the quotient $\E_2^{\sigma_1,\:\sigma_0}/S_2$ consists of $\lfloor \frac{n}{2} \rfloor$ punctured lines.
\end{proof}

Since $\E_2^{\sigma_1,\:\sigma_1}=\E_2^{\sigma_1}-\E_2^{\sigma_1,\:\sigma_0}$, we control the $S_2$-action on this stratum due to the previous result. We summarize the $S_2$-equivariant $E$-polynomials of the spaces (\ref{eq_eigenvalueSpaces}) as follows:
$$ \begin{aligned}
    e_{S_2}(\E_2^{\sigma_1}) & = (q-1)^2T - (q-1)N, \\
    e_{S_2}(\E_2^{\sigma_1,\:\sigma_0}) & = \n(q-1)T + \nn(q-1)N, \\
    e_{S_2}(\E_2^{\sigma_1,\:\sigma_1}) & = (q-1)\left(q-\n-1\right)T - \left(n-\n\right)(q-1)N.
\end{aligned} $$

From the two partitions $\sigma_0$ and $\sigma_1$, we can consider three non-equivalent Jordan types, namely
$$ \xi_0 = (\sigma_0, \{(1),(2)\}), \qquad \xi_1 = (\sigma_0, \{(1,2)\}), \qquad \xi_2 = (\sigma_1, \{(1), (2)\}). $$
Here $\xi_0$ is the type of the diagonalizable matrices with repeated eigenvalues (i.e., $\lambda \Id_2$ for $\lambda\in\C^*$), $\xi_1$ corresponds to Jordan matrices $J_{\lambda}=\left( \begin{smallmatrix}
    \lambda & 0 \\ 1 & \lambda
\end{smallmatrix} \right)$ for $\lambda\in\C^*$ and $\xi_2$ corresponds to diagonalizable matrices with distinct eigenvalues. Therefore, the only refinements are the trivial ones $\xi_i\preceq \xi_i$ for $i=0,1,2$ and $\xi_2\preceq\xi_0$. In what follows, we shall make a systematic use of the description of the set (\ref{eq_stratumJordanTypeRepVarGL}) 
given in Proposition \ref{propo_strataRepVar}
$$ \mathfrak{R}_n(\GL_2)_{\xi',\:\xi} \cong  
	\left.\left(\J_{\xi',\:\xi} \times \big(\PGL_{2}/ \Stab(\xi')\big) \times \widetilde{\Stab}(\xi) \right) \right/ S_{\xi',\:\xi}. $$

% % % % % % % % % % % % % % % % % % % % % % % % % % % % % % % % % % % % % % % % % % % % % % %
\subsection{Stratum $\mathfrak{R}_n(\GL_2)_{\xi_0,\:\xi_0}$}\label{subsection_stratum00GL2} 
% % % % % % % % % % % % % % % % % % % % % % % % % % % % % % % % % % % % % % % % % % % % % % %

Since 
$$ \Stab(\xi_0)=\PGL_2, \qquad \widetilde{\Stab}(\xi_0)=\GL_2; $$
we have an isomorphism
\begin{equation}\label{eq_stratumGL2_00}
    \mathfrak{R}_n(\GL_2)_{\xi_0,\:\xi_0} \cong \C^*\times\GL_2.
\end{equation}
Thus, its $E$-polynomial is given by $e(\mathfrak{R}_n(\GL_2)_{\xi_0,\:\xi_0})=q(q+1)(q-1)^3$.

% % % % % % % % % % % % % % % % % % % % % % % % % % % % % % % % % % % % % % % % % % % % % % %
\subsection{Stratum $\mathfrak{R}_n(\GL_2)_{\xi_1,\:\xi_1}$}\label{subsection_stratum11GL2} 
% % % % % % % % % % % % % % % % % % % % % % % % % % % % % % % % % % % % % % % % % % % % % % %

Now we have that 
$$ \Stab(\xi_1)= \left\{ \left[ \begin{matrix}
    1 & 0 \\ y & 1
\end{matrix} \right] \right\} \cong\C, \qquad  \widetilde{\Stab}(\xi_1)=\left\{ \begin{pmatrix}
    x & 0 \\ y & x
\end{pmatrix} \right\}\cong\C^*\times\C; $$
so there is an isomorphism
\begin{equation}\label{eq_stratumGL2_11}
    \mathfrak{R}_n(\GL_2)_{\xi_1,\:\xi_1} \cong \C^* \times (\PGL_2/\C) \times (\C^*\times\C).
\end{equation}
Thus, its $E$-polynomial is given by $e(\mathfrak{R}_n(\GL_2)_{\xi_1,\:\xi_1})=q(q+1)(q-1)^3$.

% % % % % % % % % % % % % % % % % % % % % % % % % % % % % % % % % % % % % % % % % % % % % % %
\subsection{Stratum $\mathfrak{R}_n(\GL_2)_{\xi_2,\:\xi_2}$}\label{subsection_stratum22GL2}
% % % % % % % % % % % % % % % % % % % % % % % % % % % % % % % % % % % % % % % % % % % % % % %

In this case, we have that
$$ \Stab(\xi_2)= \left\{ \left[ \begin{matrix}
    x & 0 \\ 0 & y
\end{matrix} \right] \right\} \cong\mathcal{D}, \qquad  \widetilde{\Stab}(\xi_2)=\left\{ \begin{pmatrix}
    x & 0 \\ 0 & y
\end{pmatrix} \right\}\cong (\C^*)^2; $$
where $\mathcal{D}$ denotes the space of diagonal matrices in $\PGL_2$. In this case, we can write
\begin{equation}\label{eq_stratumGL2_22}
    \mathfrak{R}_n(\GL_2)_{\xi_2,\:\xi_2} \cong R_{\xi_2,\:\xi_2} / S_2,
\end{equation}
where $R_{\xi_2,\:\xi_2}=\E_{2}^{\sigma_1,\:\sigma_1} \times \:  (\PGL_2/\mathcal{D}) \times (\C^*)^2$. Since we control the $S_2$-action on $\E_{2}^{\sigma_1,\:\sigma_1}$ and $(\C^*)^2$, we have that
$$ \begin{aligned}
    e_{S_2}(\mathfrak{R}_n(\GL_2)_{\xi_2,\:\xi_2}) & = e_{S_2}(\E_{2}^{\sigma_1,\:\sigma_1}) \otimes e_{S_2}(\PGL_2/\mathcal{D}) \otimes e_{S_2}((\C^*)^2) \\
    & = q(q+1)(q-1)^3 \left( \left(q-\left\lfloor\frac{n}{2}\right\rfloor-1\right) T - \left(n-\left\lfloor\frac{n}{2}\right\rfloor\right) N \right).
\end{aligned} $$
In particular, $e(\mathfrak{R}_n(\GL_2)_{\xi_2,\:\xi_2})=q(q+1)(q-1)^3\left(q-\left\lfloor\frac{n}{2}\right\rfloor-1\right)$.

% % % % % % % % % % % % % % % % % % % % % % % % % % % % % % % % % % % % % % % % % % % % % % %
\subsection{Stratum $\mathfrak{R}_n(\GL_2)_{\xi_2,\:\xi_0}$} \label{subsection_stratum20GL2}
% % % % % % % % % % % % % % % % % % % % % % % % % % % % % % % % % % % % % % % % % % % % % % %

Finally, we have that
\begin{equation}\label{eq_stratumGL2_20}
    \mathfrak{R}_n(\GL_2)_{\xi_2,\:\xi_0} \cong R_{\xi_2,\:\xi_0} / S_2,
\end{equation}
where $R_{\xi_2,\:\xi_0}=\E_2^{\sigma_1,\:\sigma_0}\times (\PGL_2/\mathcal{D}) \times \GL_2$. We already know the $S_2$-equivariant $E$-polynomial of the two first factors. For the third one, we could make use of Lemma \ref{lemma_finiteGroupAction}. Therefore, we have that
$$ \begin{aligned}
    e_{S_2}(\mathfrak{R}_n(\GL_2)_{\xi_2,\:\xi_0}) & = e_{S_2}(\E_{2}^{\sigma_1,\:\sigma_0}) \otimes e_{S_2}(\PGL_2/\mathcal{D}) \otimes e_{S_2}(\GL_2) \\
    & = q^2(q+1)(q-1)^3\left( \left( \n q + \nn \right) T + \left( \nn q + \n \right) N \right).
\end{aligned} $$
In particular, $e(\mathfrak{R}_n(\GL_2)_{\xi_2,\:\xi_0})=q^2(q+1)(q-1)^3\left( \n q + \nn \right)$.

To sum up, the $E$-polynomial of the $\GL_2$-representation variety of the $n$-twisted Hopf link is
$$ \begin{aligned}
    e(\mathfrak{R}_n(\GL_2)) & = \sum_{i=0}^2 e(\mathfrak{R}_n(\GL_2)_{\xi_i\:,\:\xi_i}) + e(\mathfrak{R}_n(\GL_2)_{\xi_2,\:\xi_0}) \\
    & = q(q+1)(q-1)^3\left( \left\lfloor\frac{n}{2}\right\rfloor q^2 + \left(n-\left\lfloor\frac{n}{2}\right\rfloor\right)q - \left\lfloor\frac{n}{2}\right\rfloor + 1 \right).
\end{aligned} $$
In particular, following Remark \ref{remark_equivariantMotives}, we obtain the following result.
\begin{proposition}
    The motive of the $\GL_2$-representation variety of the $n$-twisted Hopf link is
    $$ [\mathfrak{R}_n(\GL_2)]= q(q+1)(q-1)^3\left( \npeq q^2 + \left(n-\npeq\right)q - \npeq + 1 \right). $$
\end{proposition}

% % % % % % % % % % % % % % % % % % % % % % % % % % % % % % % % % % % % % % % % % % % % % % %
\subsection{The irreducible locus $\mathfrak{R}_n^{\mathrm{irr}}(\GL_2)$}\label{subsection_irreducibleLocusRepVar} 
% % % % % % % % % % % % % % % % % % % % % % % % % % % % % % % % % % % % % % % % % % % % % % %

In view of Remark \ref{remark_irreducibleLocus}, we only need to consider the refinement $\xi_2\preceq\xi_0$. That is, we have that
$$ \mathfrak{R}_n^{\mathrm{irr}}(\GL_2)=\mathfrak{R}_n^{\mathrm{irr}}(\GL_2)_{\xi_2,\:\xi_0}. $$

Let $\rho=(A,B)\in\mathfrak{R}_n(\GL_2)_{\xi_2,\:\xi_0}$ be an irreducible representation. There exists a basis $\mathcal{B}$ of $\C^2$ such that $A$ can be written as
$$ A=\begin{pmatrix}
    \lambda & 0 \\ 0 & \epsilon\lambda
\end{pmatrix}, $$
where $(\lambda,\epsilon)\in\C^*\times\mu_n^*$. The irreducibility of $\rho$ implies that $A$ and $B$ cannot share eigenvectors, so $B$ must belong to the set
$$ \left\{ \begin{pmatrix}
    a & b \\
    c & d
\end{pmatrix}\in\GL_2 \:\Big|\: bc\neq 0 \right\} = \GL_2 - (\C^*)^2 \times \{(b,c)\in\C^2\:|\:bc=0\}. $$
There is a residual action of $\Stab(A)\cong\mathcal{D}$ on these matrices given by $(b,c)\mapsto (\nu^2 b, \nu^{-2}c)$, so the product $k=bc$ is $\mathcal{D}$-invariant. Thus, the space $U=\{(a,d,k)\in\C^3\:|\:k\neq 0,ad\neq k\}$ parametrizes these matrices. Notice that, for the $S_2$-action on $U$ given by $(a,d,k)\mapsto (d,a,k)$, we have an isomorphism $U/S_2\cong\{ (s_1,s_2,k)\in\C^3\:|\:s_2\neq k,k\neq 0 \}$. Consider the algebraic variety 
\begin{equation}\label{eq_varietyX}
    X=\mathcal{E}_2^{\sigma_1,\:\sigma_0}\times U\times \PGL_2,
\end{equation}
which consists of $n-1$ disjoint irreducible components $\{X_{\epsilon}\:|\:\epsilon\in\mu_n^*\}$, each of which is isomorphic to $\C^*\times U\times\PGL_2$. As in Lemma \ref{lemma_actionS2},
when $n$ is odd, then $\epsilon\in\mu_n^*$ satisfies $\epsilon\neq-1$ and the $S_2$-action on $\mathcal{E}_2^{\sigma_1,\:\sigma_0}$ swaps the irreducible components $X_{\epsilon}$ and $X_{\epsilon^{-1}}$. When $n$ is even, besides the previous components, the irreducible component $X_{-1}$ is $S_2$-invariant, where the $S_2$-action is explicitly given by $(\lambda,a,d,k,P)\mapsto (-\lambda,d,a,k,P_0P)$, where $P_0=\left(\begin{smallmatrix}
    0 & 1 \\ 1 & 0
\end{smallmatrix} \right)$. Thus, the $S_2$-equivariant $E$-polynomial of $X_{-1}$ is given by
$$ \begin{aligned}
    e_{S_2}(X_{-1}) & = e_{S_2}(\C^*)\otimes e_{S_2}(U)\otimes e_{S_2}(\PGL_2) \\
    & = \left( (q-1)T \right) \otimes (q(q-1)^2T+(q-1)N) \otimes \left( q(q+1)(q-1)T \right) \\
    & = q(q+1)(q-1)^3\left( q(q-1)T + N \right),
\end{aligned} $$
so $e(X_{-1}/S_2)=q(q+1)(q-1)^3(q^2-q)$. We thus obtain the following geometrical description of the irreducible locus.

\begin{proposition}\label{propo_structureIrreducibleLocus}
    The irreducible locus $\mathfrak{R}_n^{\mathrm{irr}}(\GL_2)$ is stratified into the following components:
    \begin{itemize}
        \item when $n$ is odd, it consists of $\frac{n-1}{2}$ disjoint irreducible components, each of which is isomorphic to $\C^*\times U\times\PGL_2$;
        \item when $n$ is even, it consists of $\frac{n-2}{2}$ disjoint irreducible components, each of which is isomorphic to $\C^*\times U\times\PGL_2$, together with the additional irreducible component $X_{-1}/S_2$.
    \end{itemize}
\end{proposition}

As a consequence of Proposition \ref{propo_structureIrreducibleLocus}, the motive of the irreducible locus $\mathfrak{R}_n^{\mathrm{irr}}(\GL_2)$ is
$$ [\mathfrak{R}_n^{\mathrm{irr}}(\GL_2)] = q(q+1)(q-1)^3\left( \n q^2-\n q+\nn \right). $$
Consequently, the motive of the reducible locus is given by
$$ \begin{aligned}
    [\mathfrak{R}_n^{\mathrm{red}}(\GL_2)] & = q(q+1)(q-1)^3(nq-n+2).
\end{aligned} $$

%%%%%%%%%%%%%%%%%%%%%%%%%%%%%%%%%%%%%%%%%%%%%%%%%%%%%%%%%%%%%%%%%%%%%%%%%%%%%%%%%%%%
\section{The $\mathrm{GL}_2(\mathbb{C})$-character variety of the twisted Hopf link}\label{section_3}
%%%%%%%%%%%%%%%%%%%%%%%%%%%%%%%%%%%%%%%%%%%%%%%%%%%%%%%%%%%%%%%%%%%%%%%%%%%%%%%%%%%%

In this section, we describe the $\GL_2$-character variety of the twisted Hopf link. In particular, we make use of the stratification (\ref{eq_stratificationRedIrr}) of the $\GL_2$-representation variety into its irreducible and reducible loci.

% % % % % % % % % % % % % % % % % % % % % % % % % % % % % % % % % % % %
\subsection{The reducible locus $\mathfrak{M}^{\mathrm{red}}_n(\GL_2)$}
% % % % % % % % % % % % % % % % % % % % % % % % % % % % % % % % % % % %

A pair of matrices $(A,B)\in\mathfrak{R}_n^{\mathrm{red}}(\GL_2)$ is $S$-equivalent to a pair of matrices of the form
$$
  \left(\left(\begin{array}{cc} \lambda_1 & 0 \\ 0 & \lambda_2 \end{array} \right) , \left(\begin{array}{cc} \kappa_1 & 0 \\ 0 & \kappa_2 \end{array} \right) \right),
$$
where $(\lambda_1,\lambda_2,\kappa_1,\kappa_2) \in (\C^*)^4$. These pairs are defined up to the $S_2$-action given by
$$ (\lambda_1,\lambda_2,\kappa_1,\kappa_2) \mapsto (\lambda_2,\lambda_1,\kappa_2,\kappa_1). $$
Since the coordinate-swapping $S_2$-action on $(\C^*)^2$ was already considered and we know that $(\C^*)^2 / S_2 \cong \C \times \C^*$, we deduce that
$$ \begin{aligned}
    e_{S_2}(\mathfrak{M}^{\mathrm{red}}_n(\GL_2)) & = e_{S_2}((\C^*)^2)^2 = \left((q^2-q)T-(q-1)N\right)^2 \\
    & = (q^2+1)(q-1)^2 T - 2q(q-1)^2 N. 
\end{aligned} $$
In particular, this yields
$e(\mathfrak{M}^{\mathrm{red}}_n(\GL_2))=(q-1)^2(q^2+1)$.

% % % % % % % % % % % % % % % % % % % % % % % % % % % % % % % % % % % % %
\subsection{The irreducible locus $\mathfrak{M}^{\mathrm{irr}}_n(\GL_2)$}
% % % % % % % % % % % % % % % % % % % % % % % % % % % % % % % % % % % % %

Since the $\PGL_2$-conjugation action on $\mathfrak{R}_n^{\mathrm{irr}}(\GL_2)$ is free, we have an equality $\mathfrak{M}_n^{\mathrm{irr}}(\GL_2)=\mathfrak{R}_n^{\mathrm{irr}}(\GL_2)/\PGL_2$. Moreover, the said $\PGL_2$-action on the space (\ref{eq_varietyX}) is also free, so we have an isomorphism
$$ X/\PGL_2 \cong \mathcal{E}_2^{\sigma_1,\:\sigma_0}\times U. $$
As in Subsection \ref{subsection_irreducibleLocusRepVar}, the space $X/\PGL_2$ consists of $n-1$ disjoint irreducible components each of which is isomorphic to $\C^*\times U$. Similar to the geometrical description of the irreducible locus $\mathfrak{R}_n^{\mathrm{irr}}(\GL_2)$ given in Proposition \ref{propo_structureIrreducibleLocus}, we have an analogous result for the irreducible locus of the $\GL_2$-character variety, so its $E$-polynomial is given by
$$ \begin{aligned}
    e(\mathfrak{M}_n^{\mathrm{irr}}(\GL_2))=(q-1)^2\left(\n q^2 - \n q + \nn\right),
\end{aligned} $$
which agrees with the quotient $e(\mathfrak{R}_n^{\mathrm{irr}}(\GL_2))/e(\PGL_2)$.

To sum up, the $E$-polynomial of the $\GL_2$-character variety of the twisted Hopf link $H_n$ is
$$ \begin{aligned}
    e(\mathfrak{M}_n(\GL_2)) & = e(\mathfrak{M}_n^{\mathrm{red}}(\GL_2)) + e(\mathfrak{M}_n^{\mathrm{irr}}(\GL_2)) \\
    & = (q-1)^2(q^2+1) + (q-1)^2\left(\n q^2 - \n q + \nn\right) \\
    & = (q-1)^2\left( \left( \n +1 \right)q^2 -\n q + n-\n \right).
\end{aligned} $$

As before, by Remark \ref{remark_equivariantMotives}, we deduce the following result.
\begin{proposition}\label{prop_motiveGL2CharVar}
    The motive of the $\GL_2$-character variety of the $n$-twisted Hopf link is
    $$ [\mathfrak{M}_n(\GL_2)] = (q-1)^2\left( \left( \npeq +1 \right)q^2 -\npeq q + n-\npeq \right). $$
\end{proposition}

%%%%%%%%%%%%%%%%%%%%%%%%%%%%%%%%%%%%%%%%%%%%%%%%%%%%%%%%%%%%%%%%%%%%%%%%%%%%%%%%%%%%%%%%%%
\section{The $\mathrm{AGL}_r(\mathbb{C})$-representation variety of the twisted Hopf link}\label{section_4}
%%%%%%%%%%%%%%%%%%%%%%%%%%%%%%%%%%%%%%%%%%%%%%%%%%%%%%%%%%%%%%%%%%%%%%%%%%%%%%%%%%%%%%%%%%

We now turn our attention to the $\AGL_r$-representation variety of the twisted Hopf link $H_n$, where $\AGL_r$ denotes the group of affine transformations of the affine space $\C^r$. To this end, we will make use of the stratification of the $\GL_r$-representation variety of $H_n$ given in Section \ref{section_1}. 

Given a pair $(A,B)\in\mathfrak{R}_n(\AGL_r)$, we know that $A$ and $B$ can be put in the form
$$ A=\begin{pmatrix}
    1 & 0 \\
    \alpha & A_0
\end{pmatrix}, \quad B=\begin{pmatrix}
    1 & 0 \\
    \beta & B_0
\end{pmatrix}; $$
where $\alpha,\beta\in\C^r$ and $A_0,B_0\in\GL_r$. If we denote by $\Phi_k$ the polynomial 
$$ \Phi_k(x)=1+x+x^2+\dots+x^{k-1}=\frac{x^k-1}{x-1}\in\C[x], $$
a direct computation shows that
$$ A^n = \begin{pmatrix}
    1 & 0 \\
    \Phi_n(A_0)\alpha & A_0^n
\end{pmatrix}, $$
so the condition $[A^n,B]=\Id_{r+1}$ is equivalent to the equality
$$ \begin{pmatrix}
    1 & 0 \\
    \Phi_n(A_0)\alpha+A_0^n\beta & A_0^nB_0
\end{pmatrix} = \begin{pmatrix}
    1 & 0 \\
    B_0\Phi_n(A_0)\alpha+\beta & B_0A_0^n
\end{pmatrix}. $$
As a byproduct, the $\AGL_r$-representation variety of $H_n$ can be described as
\begin{equation}\label{eq_representationVarietyAGL}
    \mathfrak{R}_n(\AGL_r) = \left\{ (\alpha,\beta,A_0,B_0) \in (\C^r)^2\times\GL_r^2 \:\Bigg|\: {
    \begin{aligned}
    [A_0^n,B_0] &= \Id_r \\
    (B_0-\Id_r)\Phi_n(A_0)\alpha &= (A_0^n-\Id_r)\beta 
    \end{aligned}} \right\}.
\end{equation}

There is a surjective morphism of algebraic varieties
\begin{equation}\label{eq_alexanderFibration}
    \pi_r : \mathfrak{R}_n(\AGL_r) \longrightarrow \mathfrak{R}_n(\GL_r), \quad (\alpha,\beta,A_0,B_0)\mapsto (A_0,B_0);
\end{equation}
which is referred to as the \textit{Alexander fibration} for $r=1$ in \cite{gonzalezMartinezMunoz25}, which is induced by the natural surjective morphism $\AGL_r\to\GL_r$. The fiber of this morphism at a point $(A_0,B_0)\in\mathfrak{R}_n(\GL_r)$ is the kernel of the linear map
\begin{equation}\label{eq_linearMapAGL}
    T_{(A_0,B_0)}:\C^r \times \C^r \longrightarrow \C^r, \quad (\alpha,\beta) \mapsto (B_0-\Id_r)\Phi_n(A_0)\alpha - (A_0^n-\Id_r)\beta.
\end{equation}
When the context is clear, we shall omit the subscript $(A_0,B_0)$.

\begin{remark}
    Note that if $(A_0,B_0),(A'_0,B'_0)\in\GL_r^2$ lie in the same $\GL_r$-orbit by the conjugation action, then the linear map ($\ref{eq_linearMapAGL}$) satisfies $\rank(T_{(A_0,B_0)})=\rank(T_{(A'_0,B'_0)})$.
\end{remark}

The morphism $\pi_r$ induces a natural stratification of the variety $\mathfrak{R}_n(\AGL_r)$, obtained from the stratification (\ref{eq_stratificationRepVarGLr}),
\begin{equation}\label{eq_stratificationAGL}
    \mathfrak{R}_n(\AGL_r) = \bigsqcup_{\xi',\:\xi}\pi_r^{-1}\left( \mathfrak{R}_n(\GL_r)_{\xi',\:\xi} \right),
\end{equation}
where $\xi'$ and $\xi$ range over a set of non-equivalent representatives of all Jordan types of rank $r$.

\begin{remark}\label{remark_charVarAGLR}
    Since the affine group $\AGL_r$ is not reductive, we cannot directly apply Nagata's theorem to obtain the affineness of the $\AGL_r$-character variety. However, it can be shown that, for any finitely generated group $\Gamma$, the subalgebras of invariant functions satisfy
    $$ \mathcal{O}(\mathfrak{R}(\Gamma,\AGL_r))^{\AGL_r} = \mathcal{O}(\mathfrak{R}(\Gamma,\GL_r))^{\GL_r}. $$
    Therefore, the $\AGL_r$-character variety of $\Gamma$ is an affine variety, and moreover,
    $$ \mathfrak{M}(\Gamma,\AGL_r)=\mathfrak{M}(\Gamma,\GL_r). $$
    For further discussion and a proof of the aforementioned equality, see \cite[Section 5]{gonzalezLogaresMunoz23}.
\end{remark}

This remark, together with %the developments of Section \ref{section_3} and
Proposition \ref{prop_motiveGL2CharVar}, completes the proof of Theorem \ref{intro-thmB}.

% % % % % % % % % % % % % %
\subsection{The case $r=1$}
% % % % % % % % % % % % % %

Let us analyze the representation variety of $H_n$ for $\AGL_1$. In this case
$$ \mathfrak{R}_n(\GL_1)=(\C^*)^2, $$
so the description (\ref{eq_representationVarietyAGL}) of the $\AGL_1$-representation variety is given by
$$ \mathfrak{R}_n(\AGL_1) = \left\{ (\alpha,\beta,a_0,b_0) \in \C^2\times(\C^*)^2 \:|\: (b_0-1)\Phi_n(a_0)\alpha = (a_0^n-1)\beta  \right\}. $$

Given a point $(a_0,b_0)\in(\C^*)^2$, the fiber $\pi_1^{-1}(a_0,b_0)$ is the orthogonal complement of the vector
$$ ((b_0-1)\Phi_n(a_0),a_0^n-1)=((b_0-1)\Phi_n(a_0),(a_0-1)\Phi_n(a_0))\in\C^2 $$
with respect to the euclidean inner product in $\C^2$. Therefore, we have that
$$ \pi_1^{-1}(a_0,b_0) = \left\{ \begin{array}{lcc} \C & \text{if} & ((b_0-1)\Phi_n(a_0),(a_0-1)\Phi_n(a_0)) \neq (0,0) \\ \C^2 & \text{if} & ((b_0-1)\Phi_n(a_0),(a_0-1)\Phi_n(a_0)) = (0,0) \end{array} \right.. $$
Note that, since $\Phi_n(a_0)=0$ is equivalent to $a_0\in\mu_n^*$, the morphism $\pi_1$ induces three locally trivial fibrations with respect to the Zariski topology, namely:
$$ \begin{aligned}
    \C^2 \longrightarrow & \: \mathfrak{R}_n^{(0)}(\AGL_1) \longrightarrow \mu_n^*\times\C^*, \\
    \C^2 \longrightarrow & \: \mathfrak{R}_n^{(1)}(\AGL_1) \longrightarrow \{(1,1)\}, \\
    \C \longrightarrow & \: \mathfrak{R}_n^{(2)}(\AGL_1) \longrightarrow ((\C^*-\mu_n^*)\times\C^*)-\{(1,1)\};
\end{aligned} $$
where $\mathfrak{R}_n(\AGL_1)=\bigsqcup_{i=0}^2\mathfrak{R}_n^{(i)}(\AGL_1)$. Thus, we have that
$$ \begin{aligned}
    [\mathfrak{R}_n(\AGL_1)] & = \sum_{i=0}^2[\mathfrak{R}_n^{(i)}(\AGL_1)] \\
    & = [\mu_n^*\times\C^*][\C^2] + [\{(1,1)\}][\C^2] + [((\C^*-\mu_n^*)\times\C^*)-\{(1,1)\}][\C] \\
    & = (n-1)(q-1)q^2 + q^2 + (q^2-(n+1)q+(n-1))q  \\
    & = q(q-1)(nq-n+1).
\end{aligned} $$

Therefore, we obtain the following result.

\begin{theorem}\label{thm_motiveAGL1}
    The motive of the $\AGL_1$-representation variety of the $n$-twisted Hopf link is
    $$ [\mathfrak{R}_n(\AGL_1)]=q(q-1)(nq-n+1). $$
\end{theorem}

\begin{remark}
    Since the case $n=1$ corresponds to the $\AGL_1$-representation variety of the torus $T^2=\mathbb{S}^1\times\mathbb{S}^1$, the above formula agrees with that given in \cite[Thm. 3.1]{gonzalezLogaresMunoz21} for genus $g=1$.
\end{remark}

%%%%%%%%%%%%%%%%%%%%%%%%%%%%%%%%%%%%%%%%%%%%%%%%%%%%%%%%%%%%%%%%%%%%%%%%%%%%%%%%%%%%%%%%%%%%%%%%%%%%%%%%%%%%
\section{Stratification of the $\mathrm{AGL}_2(\mathbb{C})$-representation variety of the twisted Hopf link}\label{section_5}
%%%%%%%%%%%%%%%%%%%%%%%%%%%%%%%%%%%%%%%%%%%%%%%%%%%%%%%%%%%%%%%%%%%%%%%%%%%%%%%%%%%%%%%%%%%%%%%%%%%%%%%%%%%%

For the $\AGL_2$-representation variety, we are going to make use of the stratification (\ref{eq_stratificationAGL}) presented in the previous section, which relies on the stratification (\ref{eq_stratificationRepVarGLr}) of the $\GL_2$-representation variety of the $n$-twisted Hopf link presented in Section \ref{section_2}. Given two Jordan types $\xi_j\preceq \xi_i$ of rank $2$, we denote 
$$ \mathfrak{R}_n(\AGL_2)_{\xi_j\:,\:\xi_i}=\pi_2^{-1}(\mathfrak{R}_n(\GL_2)_{\xi_j\:,\:\xi_i}). $$ 
In this case, the stratification (\ref{eq_stratificationAGL}) is composed of four strata.

To begin with, we define an algebraic subvariety of $\GL_2$ that will be useful in what follows. Let us consider the algebraic variety
\begin{equation}\label{variety_omega}
    \Omega=\{B\in\GL_2\:|\:\det(B-\Id_2)=0\}
\end{equation}
consisting of matrices whose spectrum contains $1$. Moreover, set $\Omega^{\times}=\Omega-\{\Id_2\}$.

\begin{lemma}\label{lemma_varietyTheta}
    The motive of $\Omega$ is $[\Omega]=q^3-2q$.
\end{lemma}
\begin{proof}
    We can decompose $\Omega$ into three strata according to the Jordan form of its elements:
    \begin{equation*}\label{eq_stratificationOmega}
        \begin{aligned}
        \Omega^{(0)} & = \left\{ B\in\GL_2 \:\Bigg|\: B=P\begin{pmatrix}
            1 & 0 \\ 0 & 1
        \end{pmatrix}P^{-1}, \: P\in\GL_2 \right\} = \{\Id_2\}, \\
        \Omega^{(1)} & = \left\{ B\in\GL_2 \:\Bigg|\: B=P\begin{pmatrix}
            1 & 0 \\ 0 & \gamma
        \end{pmatrix}P^{-1}, \: \gamma\neq 1, \: P\in\GL_2 \right\} \cong (\C^*-\{1\})\times(\PGL_2/\mathcal{D}), \\
        \Omega^{(2)} & = \left\{ B\in\GL_2 \:\Bigg|\: B=P \begin{pmatrix}
            1 & 0 \\ 1 & 1
        \end{pmatrix}P^{-1}, \: P\in\GL_2 \right\}\cong \PGL_2/\C.
        \end{aligned}
    \end{equation*}
    It follows that $[\Omega]=1+(q-2)(q^2+q)+q^2-1=q^3-2q$, as desired.
\end{proof}

% % % % % % % % % % % % % % % % % % % % % % % % % % % % % % % %
\subsection{Stratum $\mathfrak{R}_n(\AGL_2)_{\xi_0,\:\xi_0}$}
% % % % % % % % % % % % % % % % % % % % % % % % % % % % % % % %

Given a pair of matrices $(A,B)\in\mathfrak{R}_n(\AGL_2)_{\xi_0,\:\xi_0}$, we have 
$$ \pi_2(A,B)=(\lambda\Id_2,B_0), $$
where $\lambda\in\C^*$ and $B_0\in\GL_2$. In this case, the linear map $T$ takes the form
$$ T(\alpha,\beta) = \Phi_n(\lambda)\left( (B_0-\Id_2)\alpha - (\lambda-1)\beta \right). $$

Since the equation $\Phi_n(\lambda)=0$ is equivalent to $\lambda\in\mu_n^*$, the morphism $\pi_2$ induces three fibrations corresponding to the possible values of $\lambda\in\C^*$, namely:
$$ \begin{aligned}
    p_0 & : \mathfrak{R}_n^{(0)}(\AGL_2)_{\xi_0,\:\xi_0} \longrightarrow \mu_n^*\times\GL_2, \\
    p_1 & : \mathfrak{R}_n^{(1)}(\AGL_2)_{\xi_0,\:\xi_0} \longrightarrow \{1\}\times\GL_2, \\
    p_2 & : \mathfrak{R}_n^{(2)}(\AGL_2)_{\xi_0,\:\xi_0} \longrightarrow (\C^*-\mu_n)\times\GL_2;
\end{aligned} $$
where $\mathfrak{R}_n(\AGL_2)_{\xi_0,\:\xi_0}=\bigsqcup_{i=0}^2\mathfrak{R}_n^{(i)}(\AGL_2)_{\xi_0,\:\xi_0}$ and the fibers of the maps $p_0$, $p_1$ and $p_2$ are equal to $\mathrm{Ker}(T)$.

\subsubsection*{$\bullet$ Substratum $\mathfrak{R}^{(0)}_n(\AGL_2)_{\xi_0,\:\xi_0}$}

In this case, the fibration $p_0$ is locally trivial with respect to the Zariski topology. Moreover, since $\Phi_n(\lambda)=0$, we have that $T$ is identically zero, so $\Ker(T)\cong \C^4$. Therefore, we have that
$$ \begin{aligned}
    [\mathfrak{R}_n^{(0)}(\AGL_2)_{\xi_0,\:\xi_0}] & = [\mu_n^*\times \GL_2][\C^4] = (n-1)q^5(q+1)(q-1)^2.
\end{aligned} $$

\subsubsection*{$\bullet$ Substratum $\mathfrak{R}^{(1)}_n(\AGL_2)_{\xi_0,\:\xi_0}$}

In this case, the linear map $T$ takes the form
$$ T(\alpha,\beta) = n(B_0-\Id_2)\alpha. $$
Consequently, the fibers of the map $p_1$ depend on the rank of the matrix $B_0-\Id_2$. This leads to three locally trivial fibrations with respect to the Zariski topology, according to whether $B_0=\Id_2$, $B_0\in\Omega^{\times}$ or $B_0\not\in \Omega$. Hence, we obtain
$$ \begin{aligned}
    [\mathfrak{R}_n^{(1)}(\AGL_2)_{\xi_0,\:\xi_0}] & = [\{1\}\times\{\Id_2\}][\C^4] + [\{1\}\times\Omega^\times][\C^3] + [\{1\}\times(\GL_2-\Omega)][\C^2] \\
    & = q^4 + (q+1)(q^2-q-1)q^3 + q(q^3-2q^2-q+3)q^2 = 2q^3(q+1)(q-1)^2.
\end{aligned} $$

\subsubsection*{$\bullet$ Substratum $\mathfrak{R}^{(2)}_n(\AGL_2)_{\xi_0,\:\xi_0}$}

For the Zariski locally trivial fibration $p_2$, since $\lambda\not\in\mu_n$, we have that $\rank(T)=2$, so $\Ker(T)\cong\C^2$. We thus obtain
$$ \begin{aligned}
    [\mathfrak{R}_n^{(2)}(\AGL_2)_{\xi_0,\:\xi_0}] & = [(\C^*-\mu_n)\times \GL_2][\C^2] = q^3(q+1)(q-1)^2(q-n-1).
\end{aligned} $$

Adding all the contributions, the motive of $\mathfrak{R}_n(\AGL_2)_{\xi_0,\:\xi_0}$ is
\begin{equation}\label{eq_motiveAGL2_00}
    \begin{aligned}
    [\mathfrak{R}_n(\AGL_2)_{\xi_0,\:\xi_0}] & = \sum_{i=0}^2 [\mathfrak{R}_n^{(i)}(\AGL_2)_{\xi_0,\:\xi_0}] \\
    & = q^3(q+1)(q-1)^2\left((n-1)q^2+q-n+1\right).
\end{aligned}
\end{equation}

% % % % % % % % % % % % % % % % % % % % % % % % % % % % % % % %
\subsection{Stratum $\mathfrak{R}_n(\AGL_2)_{\xi_1,\:\xi_1}$}
% % % % % % % % % % % % % % % % % % % % % % % % % % % % % % % %

Given a pair of matrices $(A,B)\in\mathfrak{R}_n(\AGL_2)_{\xi_1,\:\xi_1}$, there exists a basis $\mathcal{B}$ of $\C^2$ in which 
$$ \pi_2(A,B) = \left( \begin{pmatrix}
    \lambda & 0 \\ 1 & \lambda
\end{pmatrix},\begin{pmatrix}
    \kappa & 0 \\
    y & \kappa
\end{pmatrix} \right), $$
where $\lambda\in\C^*$ and $(\kappa,y)\in\C^*\times\C$. With respect to the basis $\mathcal{B}$, the linear map $T$ takes the form
$$ T(\alpha,\beta)=\begin{pmatrix}
    \Phi_n(\lambda)(\kappa-1) & 0 \\
    \Phi_n(\lambda)y+\Phi'_n(\lambda)(\kappa-1) & \Phi_n(\lambda)(\kappa-1)
\end{pmatrix}\alpha - \begin{pmatrix}
    \lambda^n-1 & 0 \\
    n\lambda^{n-1} & \lambda^n-1
\end{pmatrix}\beta. $$

As before, the morphism $\pi_2$ induces three fibrations according to the possible values of $\lambda$:
$$ \begin{aligned}
    q_0 & : \mathfrak{R}_n^{(0)}(\AGL_2)_{\xi_1,\:\xi_1} \longrightarrow \mu_n^* \times (\PGL_2/\C) \times \C \times \C^*, \\
    q_1 & : \mathfrak{R}_n^{(1)}(\AGL_2)_{\xi_1,\:\xi_1} \longrightarrow \{1\} \times (\PGL_2/\C) \times \C \times \C^*, \\
    q_2 & : \mathfrak{R}_n^{(2)}(\AGL_2)_{\xi_1,\:\xi_1} \longrightarrow (\C^*-\mu_n) \times (\PGL_2/\C) \times \C \times \C^*;
\end{aligned} $$
where $\mathfrak{R}_n(\AGL_2)_{\xi_1,\:\xi_1}=\bigsqcup_{i=0}^2\mathfrak{R}_n^{(i)}(\AGL_2)_{\xi_1,\:\xi_1}$ and the fibers of the maps $q_0$, $q_1$ and $q_2$ are equal to $\mathrm{Ker}(T)$.

\subsubsection*{$\bullet$ Substratum $\mathfrak{R}^{(0)}_n(\AGL_2)_{\xi_1,\:\xi_1}$} 

For the locally trivial fibration $q_0$, the condition $\Phi_n(\lambda)=0$ implies that $\rank(T)=1$, so $\Ker(T)\cong\C^3$. Hence,
$$ \begin{aligned}
    [\mathfrak{R}_n^{(0)}(\AGL_2)_{\xi_1,\:\xi_1}] & = [\mu_n^*\times (\PGL_2/\C) \times \C \times \C^*][\C^3] = (n-1)q^4(q+1)(q-1)^2.
\end{aligned} $$

\subsubsection*{$\bullet$ Substratum $\mathfrak{R}^{(1)}_n(\AGL_2)_{\xi_1,\:\xi_1}$}

Here, the fibers of $q_1$ depend on the eigenvalue $\kappa$ of $B_0$, yielding two locally trivial fibrations corresponding to $\kappa=1$ and $\kappa\neq 1$. Thus,
$$ \begin{aligned}
    [\mathfrak{R}_n^{(1)}(\AGL_2)_{\xi_1,\:\xi_1}] & = [\{1\}\times(\PGL_2/\C) \times \C \times \{1\}][\C^3] \\
    & \quad + [\{1\}\times(\PGL_2/\C) \times \C\times (\C^*-\{1\})][\C^2] \\
    & = q^4(q+1)(q-1) + q^3(q+1)(q-1)(q-2) = 2q^3(q+1)(q-1)^2.
\end{aligned} $$

\subsubsection*{$\bullet$ Substratum $\mathfrak{R}^{(2)}_n(\AGL_2)_{\xi_1,\:\xi_1}$}

In this case, the fibration $q_2$ is again locally trivial. Moreover, since $\rank(T)=2$, we obtain that $\Ker(T)\cong\C^2$. Therefore, we have
$$ \begin{aligned}
    [\mathfrak{R}_n^{(2)}(\AGL_2)_{\xi_1,\:\xi_1}] & = [(\C^*-\mu_n) \times (\PGL_2/\C) \times \C \times \C^*][\C^2] = q^3(q+1)(q-1)^2(q-n-1).
\end{aligned} $$

Adding all the contributions, we thus obtain
\begin{equation}\label{eq_motiveAGL2_11}
    \begin{aligned}
        [\mathfrak{R}_n(\AGL_2)_{\xi_1,\:\xi_1}] & = \sum_{i=0}^2 [\mathfrak{R}_n^{(i)}(\AGL_2)_{\xi_1,\:\xi_1}] = q^3(q+1)(q-1)^2(nq-n+1).
    \end{aligned}
\end{equation}

% % % % % % % % % % % % % % % % % % % % % % % % % % % % % % % %
\subsection{Stratum $\mathfrak{R}_n(\AGL_2)_{\xi_2,\:\xi_2}$}
% % % % % % % % % % % % % % % % % % % % % % % % % % % % % % % %

Given a pair of matrices $(A,B)\in\mathfrak{R}_n(\AGL_2)_{\xi_2,\:\xi_2}$, there exists a basis $\mathcal{B}$ of $\C^2$ in which
$$ \pi_2(A,B) = \left( \begin{pmatrix}
    \lambda_1 & 0 \\ 0 & \lambda_2
\end{pmatrix},\begin{pmatrix}
    \kappa_1 & 0 \\
    0 & \kappa_2
\end{pmatrix} \right), $$
where $(\lambda_1,\lambda_2)\in\E_2^{\sigma_1\:,\:\sigma_1}$ and $(\kappa_1,\kappa_2)\in(\C^*)^2$. Here the linear map $T$ takes the form
$$ T(\alpha,\beta)=\begin{pmatrix}
    \Phi_n(\lambda_1)(\kappa_1-1) & 0 \\
    0 & \Phi_n(\lambda_2)(\kappa_2-1)
\end{pmatrix}\alpha - \begin{pmatrix}
    \lambda_1^n-1 & 0 \\
    0 & \lambda_2^n-1
\end{pmatrix}\beta $$
with respect to the basis $\mathcal{B}$. In this case, the morphism $\pi_2$ induces three fibrations according to the possible values of $(\lambda_1,\lambda_2)$:
$$ r_i : \mathfrak{R}_n^{(i)}(\AGL_2)_{\xi_2,\:\xi_2} \longrightarrow Y_i, \quad \text{for } i=0,1,2;$$
where $\mathfrak{R}_n(\AGL_2)_{\xi_2,\:\xi_2}=\bigsqcup_{i=0}^2\mathfrak{R}_n^{(i)}(\AGL_2)_{\xi_2,\:\xi_2}$, $\mathfrak{R}_n(\GL_2)_{\xi_2,\:\xi_2}=\bigsqcup_{i=0}^2Y_i$ is a stratification of the variety (\ref{eq_stratumGL2_22}), and the fibers of the maps $r_0$, $r_1$ and $r_2$ are equal to $\mathrm{Ker}(T)$.

\subsubsection*{$\bullet$ Substratum $\mathfrak{R}^{(0)}_n(\AGL_2)_{\xi_2,\:\xi_2}$}

This stratum corresponds to the case in which one of the eigenvalues of $A_0$ lies in $\mu_n^*$. In particular, we may assume that $\lambda_1\in\mu_n^*$. Since $(\lambda_1,\lambda_2)\in\mathcal{E}_2^{\sigma_1\:,\:\sigma_1}$, it follows that $\lambda_2\in\C^*-\mu_n$. Therefore,
$$ Y_0=\mu_n^*\times(\C^*-\mu_n)\times(\PGL_2/\mathcal{D})\times(\C^*)^2. $$
Here, since $\rank(T)=1$, it holds $\mathrm{Ker}(T)\cong\C^3$. Hence,
$$ \begin{aligned}
    [\mathfrak{R}_n^{(0)}(\AGL_2)_{\xi_2,\:\xi_2}] & = [Y_0][\C^3] = (n-1)q^4(q+1)(q-1)^2(q-n-1).
\end{aligned} $$

\subsubsection*{$\bullet$ Substratum $\mathfrak{R}^{(1)}_n(\AGL_2)_{\xi_2,\:\xi_2}$}

This stratum corresponds to the case in which one of the eigenvalues of $A_0$ is equal to $1$. In particular, we may assume that $\lambda_1=1$. As before, since $(\lambda_1,\lambda_2)\in\mathcal{E}_2^{\sigma_1\:,\:\sigma_1}$, it follows that $\lambda_2\in\C^*-\mu_n$, so
$$ Y_1=\{1\}\times(\C^*-\mu_n)\times(\PGL_2/\mathcal{D})\times(\C^*)^2. $$
In this case, the linear map $T$ takes the form
$$ T(\alpha,\beta)=\begin{pmatrix}
    n(\kappa_1-1) & 0 \\
    0 & \Phi_n(\lambda_2)(\kappa_2-1)
\end{pmatrix}\alpha - \begin{pmatrix}
    0 & 0 \\
    0 & \lambda_2^n-1
\end{pmatrix}\beta, $$
so the vector space $\mathrm{Ker}(T)$ depends on the possible values of $\kappa_1$. Indeed, the morphism $r_1$ gives rise to two locally trivial fibrations corresponding to $\kappa_1=1$ and $\kappa_1\neq 1$. Hence, we have
$$ \begin{aligned}
    [\mathfrak{R}_n^{(1)}(\AGL_2)_{\xi_2,\:\xi_2}] & = [\{1\}\times(\C^*-\mu_n)\times(\PGL_2/\mathcal{D})\times\{1\}\times\C^*][\C^3] \\
    & \quad + [\{1\}\times(\C^*-\mu_n)\times(\PGL_2/\mathcal{D})\times(\C^*-\{1\})\times\C^*][\C^2] \\
    & = q^4(q+1)(q-1)(q-n-1) + q^3(q+1)(q-1)(q-2)(q-n-1) \\
    & = 2q^3(q+1)(q-1)^2(q-n-1).
\end{aligned} $$

\subsubsection*{$\bullet$ Substratum $\mathfrak{R}^{(2)}_n(\AGL_2)_{\xi_2,\:\xi_2}$}

Lastly, this stratum corresponds to the case in which none of the eigenvalues of $A_0$ belong to $\mu_n$. Therefore, if we set
$$ \widetilde{Y_2}=\left( \mathcal{E}_2^{\sigma_1\:,\:\sigma_1}-(\mu_n\times(\C^*-\mu_n) \sqcup (\C^*-\mu_n)\times\mu_n) \right) \times (\PGL_2/\mathcal{D}) \times (\C^*)^2, $$
then $Y_2=\widetilde{Y_2}/S_2$, where $S_2$ acts as in (\ref{eq_stratumGL2_22}). Observe that
$$ \begin{aligned}
    e_{S_2}\left( \mathcal{E}_2^{\sigma_1\:,\:\sigma_1}\!-\!(\mu_n\times(\C^*\!-\!\mu_n)\! \sqcup\! (\C^*\!-\!\mu_n)\times\mu_n) \right) & = \left( q^2 - (n+\npeq+2)q + n^2+n+\npeq+1 \right)T \\
    & \quad - \left( \left(2n-\npeq\right)q - n^2 - 2n + \npeq \right)N,
\end{aligned} $$
so the $S_2$-equivariant $E$-polynomial of $\widetilde{Y_2}$ is given by
$$ \begin{aligned}
    e_{S_2}(\widetilde{Y_2}) & = e_{S_2}(\mathcal{E}_2^{\sigma_1\:,\:\sigma_1}-(\mu_n\times(\C^*-\mu_n) \sqcup (\C^*-\mu_n)\times\mu_n)) \otimes e_{S_2}(\PGL_2/\mathcal{D})\otimes e_{S_2}((\C^*)^2) \\
    & = q(q+1)(q-1)^2\left( \left( q^2-\left(n+\npeq+2\right)q+n^2+n+\npeq+1 \right)T \right. \\
    & \quad \left. - \left( \left(2n-\npeq\right)q-n^2-2n+\npeq \right)N \right).
\end{aligned} $$
Since $\rank(T)=2$, we have that $\mathrm{Ker}(T)\cong\C^2$. Therefore,
$$ \begin{aligned}
    [\mathfrak{R}_n^{(2)}(\AGL_2)_{\xi_2,\:\xi_2}] = [Y_2][\C^2] = q^3(q+1)(q-1)^2\left( q^2-\left(n+\npeq+2\right)q+n^2+n+\npeq+1 \right).
\end{aligned} $$

Adding all the contributions, we thus obtain
\begin{equation}\label{eq_motiveAGL2_22}
    \begin{aligned}
        [\mathfrak{R}_n(\AGL_2)_{\xi_2,\:\xi_2}] & = \sum_{i=0}^2 [\mathfrak{R}_n^{(i)}(\AGL_2)_{\xi_2,\:\xi_2}] \\
        & = q^3(q+1)(q-1)^2\left( nq^2 - \left(n^2+n+\npeq-1\right)q + n^2-n+\npeq-1 \right).
    \end{aligned}
\end{equation}

% % % % % % % % % % % % % % % % % % % % % % % % % % % % % % % %
\subsection{Stratum $\mathfrak{R}_n(\AGL_2)_{\xi_2,\:\xi_0}$}
% % % % % % % % % % % % % % % % % % % % % % % % % % % % % % % %

Let $(A,B)\in\mathfrak{R}_n(\AGL_2)_{\xi_2,\:\xi_0}$ and write $\pi_2(A,B)=(A_0,B_0)$. There exists a basis $\mathcal{B}$ of $\C^2$ in which
$$ A_0 = \begin{pmatrix}
    \lambda & 0 \\ 0 & \epsilon\lambda
\end{pmatrix}, $$
where $(\lambda,\epsilon)\in\C^*\times\mu_n^*$ and $B_0\in\GL_2$ is arbitrary. In this case, the linear map $T$ takes the form
$$ T(\alpha,\beta)=(B_0-\Id_2)\begin{pmatrix}
    \Phi_n(\lambda) & 0 \\
    0 & \Phi_n(\epsilon\lambda)
\end{pmatrix}\alpha - \begin{pmatrix}
    \lambda^n-1 & 0 \\
    0 & \lambda^n-1
\end{pmatrix}\beta $$
with respect to the basis $\mathcal{B}$. As in the previous cases, the morphism $\pi_2$ induces three fibrations according to the possible values of $(\lambda,\epsilon)$:
$$ s_i : \mathfrak{R}_n^{(i)}(\AGL_2)_{\xi_2,\:\xi_0} \longrightarrow Z_i, \quad \text{for }i=0,1,2; $$
where $\mathfrak{R}_n(\AGL_2)_{\xi_2,\:\xi_0}=\bigsqcup_{i=0}^2\mathfrak{R}_n^{(i)}(\AGL_2)_{\xi_2,\:\xi_0}$, $\mathfrak{R}_n(\GL_2)_{\xi_2,\:\xi_0}=\bigsqcup_{i=0}^2Z_i$ is a stratification of the variety (\ref{eq_stratumGL2_20}), and the fibers of the maps $s_0$, $s_1$ and $s_2$ are equal to $\mathrm{Ker}(T)$.

\subsubsection*{$\bullet$ Substratum $\mathfrak{R}^{(0)}_n(\AGL_2)_{\xi_2,\:\xi_0}$}

This stratum corresponds to the case in which both of the eigenvalues of $A_0$ lies in $\mu_n^*$, so we have that $(\lambda,\epsilon)\in \{(\lambda,\epsilon)\in(\mu_n^*)^2\:|\:\epsilon\lambda\neq 1\}\cong (\mu_n^*)^2-\mu_n^*$. Therefore, if we set
$$ \widetilde{Z_0} = ((\mu_n^*)^2-\mu_n^*)\times(\PGL_2/\mathcal{D})\times\GL_2, $$
then $Z_0 = \widetilde{Z_0}/S_2$, where $S_2$ acts on $(\PGL_2/\mathcal{D})\times\GL_2$ as in (\ref{eq_stratumGL2_20}), and acts freely on $(\mu_n^*)^2-\mu_n^*$ by $(\lambda,\epsilon)\mapsto (\epsilon\lambda,\epsilon^{-1})$. In particular, the quotient $((\mu_n^*)^2-\mu_n^*)/S_2$ consists of $\frac{(n-1)(n-2)}{2}$ points. Therefore, the $S_2$-equivariant $E$-polynomial of $\widetilde{Z_0}$ is given by
$$ \begin{aligned}
    e_{S_2}(\widetilde{Z_0}) & = e_{S_2}((\mu_n^*)^2-\mu_n^*)\otimes e_{S_2}(\PGL_2/\mathcal{D}) \otimes e_{S_2}(\GL_2) \\
    & = \dfrac{(n-1)(n-2)}{2}q^2(q+1)^2(q-1)^2(T+N).
\end{aligned} $$
Here $T=0$, so $\mathrm{Ker}(T)\cong\C^4$. Thus,
$$ \begin{aligned}
    [\mathfrak{R}_n^{(0)}(\AGL_2)_{\xi_2,\:\xi_0}] & = [Z_0][\C^4] = \dfrac{(n-1)(n-2)}{2}q^6(q+1)^2(q-1)^2.
\end{aligned} $$

\subsubsection*{$\bullet$ Substratum $\mathfrak{R}^{(1)}_n(\AGL_2)_{\xi_2,\:\xi_0}$}

This stratum corresponds to the case in which one of the eigenvalues of $A_0$ is equal to $1$. In particular, we may assume that $\lambda=1$. Therefore,
$$ Z_1 = \{1\}\times\mu_n^* \times (\PGL_2/\mathcal{D}) \times \GL_2. $$
In this case, the linear map $T$ takes the form
$$ T(\alpha,\beta)=(B_0-\mathrm{Id}_2)\begin{pmatrix}
    n & 0 \\
    0 & 0
\end{pmatrix}\alpha, $$
so $\mathrm{rank}(T)$ is equal to $0$ or $1$. %depending on whether the first vector $v\in\C^2$ of the basis $\mathcal{B}$ is an eigenvector of $B_0$ corresponding to the eigenvalue $1$, that is, whether $(B_0-\Id_2)v=0$. 
With respect to the basis $\mathcal{B}$, the condition $\mathrm{rank}(T)=0$ is equivalent to $B_0$ belonging to the set
\begin{equation}
    \Theta = \left\{ \begin{pmatrix}
        1 & b \\ 0 & d
    \end{pmatrix} \Bigg|\: d\neq 0 \right\}\cong\C^*\times\C,
\end{equation}
which is an algebraic subvariety of the variety $\Omega$ introduced in (\ref{variety_omega}). Hence, the morphism $s_1$ decomposes into two locally trivial fibrations according to whether $B_0\in\Theta$, so
$$ \begin{aligned}
    [\mathfrak{R}_n^{(1)}(\AGL_2)_{\xi_2,\:\xi_0}] & = [\{1\}\times\mu_n^*\times(\PGL_2/\mathcal{D})\times\Theta][\C^4] \\
    & \quad + [\{1\}\times\mu_n^*\times(\PGL_2/\mathcal{D})\times(\GL_2-\Theta)][\C^3] \\
    & = (n-1)q^6(q+1)(q-1) + (n-1)q^5(q+1)(q-1)(q^2-2) \\
    & = (n-1)q^5(q+1)(q+2)(q-1)^2.
\end{aligned} $$

\subsubsection*{$\bullet$ Substratum $\mathfrak{R}^{(2)}_n(\AGL_2)_{\xi_2,\:\xi_0}$}

This stratum corresponds to the case in which none of the eigenvalues of $A_0$ lies in $\mu_n$. In particular, we have that $(\lambda,\epsilon)\in (\C^*-\mu_n)\times\mu_n^*$. If we set
$$ \widetilde{Z_2} = (\C^*-\mu_n)\times\mu_n^* \times(\PGL_2/\mathcal{D})\times\GL_2, $$
then $Z_2 = \widetilde{Z_2}/S_2$, where $S_2$ acts on $(\PGL_2/\mathcal{D})\times\GL_2$ as in (\ref{eq_stratumGL2_20}),  and on $(\C^*-\mu_n)\times\mu_n^*$ by $(\lambda,\epsilon)\mapsto (\epsilon\lambda,\epsilon^{-1})$.

\begin{lemma}
    Consider the $S_2$-action on $(\C^*-\mu_n)\times\mu_n^*$ given by $(\lambda,\epsilon)\mapsto (\epsilon\lambda,\epsilon^{-1})$. For the quotient $((\C^*-\mu_n)\times\mu_n^*)/S_2$, we have the following:
    \begin{itemize}
        \item when $n$ is odd, it consists of $\frac{n-1}{2}$ copies of $\C^*-\mu_n$;
        \item when $n$ is even, it consists of $\frac{n-2}{2}$ copies of $\C^*-\mu_n$ together with an extra component isomorphic to $\C^*-\mu_{n/2}$.
    \end{itemize}
    As a consequence, its $E$-polynomial is given by $e(((\C^*-\mu_n)\times\mu_n^*)/S_2) = \npeq q - \nnnpeq$.
    %As a consequence, the $S_2$-equivariant $E$-polynomial is given by
    %$$ \begin{aligned}
    %    e_{S_2}((\C^*-\mu_n)\times\mu_n^*)&=\left( \npeq q - \nnnpeq \right) T + \left( \nnpeq q - \nnnnpeq \right) N.
    %\end{aligned} $$
\end{lemma}
\begin{proof}
    Similar to what happens in Lemma \ref{lemma_actionS2}, when $\epsilon\neq -1$, the $S_2$-action swaps the punctured lines $(\C^*-\mu_n)\times\{\epsilon\}$ and $(\C^*-\mu_n)\times\{\epsilon^{-1}\}$, so the quotient consists of $\frac{n-1}{2}$ copies of $\C^*-\mu_n$. When $n$ is even, apart from the previous contributions, the punctured line $(\C^*-\mu_n)\times\{-1\}$ is $S_2$-invariant and thus the quotient $((\C^*-\mu_n)\times\{-1\})/S_2$ is isomorphic to $(\C^*-\mu_{n/2})\times\{-1\}$ via the map $S_2\cdot (\lambda,-1) \mapsto (\lambda^2,-1)$.

    For the $E$-polynomial, note that when $n$ is odd the quotient contributes 
    $$ \frac{n-1}{2}(q-n-1) = \frac{n-1}{2}q - \frac{n^2-1}{2}, $$
    whereas when $n$ is even it contributes
    $$ \frac{n-2}{2}(q-n-1)+q-\frac{n}{2}-1 = \frac{n}{2}q - \frac{n^2}{2}. $$
    Therefore, the desired result follows.
\end{proof}

Consequently, the $S_2$-equivariant $E$-polynomial of $\widetilde{Z_2}$ is
$$ \begin{aligned}
    e_{S_2}(\widetilde{Z_2}) & = q^2(q+1)(q-1)^2\left( \left( \npeq q^2 + \left( \nnpeq - \nnnpeq \right)q - \nnnnpeq \right)T\right. \\
    & \quad + \left. \left( \nnpeq q^2 + \left( \npeq-\nnnnpeq \right)q - \nnnpeq \right)N \right).
\end{aligned} $$
Finally, since $\mathrm{rank}(T)=2$, we have $\mathrm{Ker}(T)\cong\C^2$. Thus,
$$ \begin{aligned}
    [\mathfrak{R}_n^{(2)}(\AGL_2)_{\xi_2,\:\xi_0}] = [Z_2][\C^2] = q^4(q+1)(q-1)^2\left( \npeq q^2 + \left( \nnpeq - \nnnpeq \right)q - \nnnnpeq \right).
\end{aligned} $$

Adding all the contributions, we have
\begin{equation}\label{eq_motiveAGL2_20}
    \begin{aligned}
        [\mathfrak{R}_n(\AGL_2)_{\xi_2,\:\xi_0}] & = \sum_{i=0}^2 [\mathfrak{R}_n^{(i)}(\AGL_2)_{\xi_2,\:\xi_0}] \\
        & = q^4(q+1)(q-1)^2\left( \tfrac{(n-1)(n-2)}{2}q^3 + \left( \tfrac{n(n-1)}{2}+\npeq \right)q^2 \right. \\ 
        &\quad  \left. + \left( 2(n-1) + \nnpeq - \nnnpeq \right)q - \nnnnpeq \right).
    \end{aligned}
\end{equation}

% % % % % % % % % % % % % % % % % % % % % % % % % %
\subsection{The motive of $\mathfrak{R}_n(\AGL_2)$}
% % % % % % % % % % % % % % % % % % % % % % % % % %

Finally, considering all the contributions (\ref{eq_motiveAGL2_00}), (\ref{eq_motiveAGL2_11}), (\ref{eq_motiveAGL2_22}) and (\ref{eq_motiveAGL2_20}) from the previous sections, we obtain the following result.

\begin{theorem}\label{thm_motiveAGL2}
    The motive of the $\AGL_2$-representation variety of the $n$-twisted Hopf link is
    $$ \begin{aligned}
        [\mathfrak{R}_n(\AGL_2)] &= q^3(q+1)(q-1)^2\left( \tfrac{(n-1)(n-2)}{2}q^4 + \left( \tfrac{n(n-1)}{2}+\npeq \right)q^3 \right. \\
        & \quad + \left. \left( 4n + \nnpeq - \nnnpeq - 3 \right)q^2 - \left( n^2+\npeq+\nnnnpeq-2 \right)q \right. \\
        & \quad + n^2-3n+\npeq+1 \Big).
    \end{aligned} $$
\end{theorem}

This result, together with Theorem \ref{thm_motiveAGL1}, completes the proof of Theorem \ref{intro-thmA}. As a consequence, for $n=1$ we obtain the motive of the $\AGL_2$-representation variety of the torus $T^2=\mathbb{S}^1\times\mathbb{S}^1$.

\begin{corollary}
    The motive of the $\AGL_2$-representation variety of the torus $T^2=\mathbb{S}^1\times\mathbb{S}^1$ is
    $$ \begin{aligned}
            [\mathfrak{R}_1(\AGL_2)] & = q^3(q+1)(q-1)^2(q^2+q-1).
        \end{aligned} $$
\end{corollary}

This corollary provides a first step towards computing the motive of the $\AGL_2$-representation variety of the fundamental group of a compact orientable surface of arbitrary genus $g$.

\end{document}